\documentclass{amsart}

\usepackage{amsmath,bm,bbm,amsthm, amssymb}
\usepackage{color}
\usepackage{dsfont}
\usepackage{comment}
\usepackage[utf8]{inputenc}

\theoremstyle{plain}
\newtheorem{theorem}{Theorem}
\newtheorem{proposition}[theorem]{Proposition}
\newtheorem{corollary}[theorem]{Corollary}
\newtheorem{lemma}[theorem]{Lemma}

\newtheorem{assumption}[theorem]{Assumption}

\theoremstyle{definition}
\newtheorem{definition}[theorem]{Definition}

\newtheorem{example}[theorem]{Example}

\theoremstyle{remark}
\newtheorem{remark}[theorem]{Remark}

\def\mc{\mathcal}
\def\CC{\mc C}
\def\Cov{\ms{Cov}}
\def\DD{\mc D}

\def\De{\Delta}
\def\E{\mathbb E}
\def\GG{\mc G}

\def\Ib{I_{\ms{b}}}

\def\R{\mathbb R}

\def\Var{\ms{Var}}
\def\Z{\mathbb Z}
\def\a{\alpha}
\def\ba{\,|\,}
\def\been{\begin{enumerate}}
\def\bee{\begin{example}}
\def\bel{\begin{lemma}}
\def\bepr{\begin{proposition}}
\def\bep{\begin{proof}}
\def\bet{\begin{theorem}}
\def\bec{\begin{corollary}}
\def\bede{\begin{definition}}
\def\ende{\end{definition}}

\def\bs{\boldsymbol}
\def\b{\beta}

\def\co{\colon}
\def\ddd{\mathfrak d}

\def\dist{\ms{dist}}

\def\d{{\rm d}}
\def\enen{\end{enumerate}}
\def\ene{\end{example}}
\def\enl{\end{lemma}}
\def\enpr{\end{proposition}}
\def\enp{\end{proof}}
\def\ent{\end{theorem}}
\def\enc{\end{corollary}}

\def\es{\varnothing}
\def\e{\varepsilon}
\def\ff{\infty}
\def\f{\frac}

\def\g{\gamma}
\def\im{\item}

\def\ms{\mathsf}

\def\one{\mathbbmss{1}}
\def\pa{\partial}
\def\r{\rho}
\def\sm{\setminus}

\def\su{\subseteq}
\def\s{\sigma}
\def\tff{\uparrow\infty}

\def\ti{\times}

\def\vp{\varphi}
\def\wt{\widetilde}
\def\w{\wedge}

\def\su{\subseteq}

\def\bt{\b_{n,\eta}}
\def\FF{\mc F}

\def\mc{\mathcal}
\def\CC{\mc C}
\def\Cov{\ms{Cov}}
\def\DD{\mc D}

\def\De{\Delta}
\def\E{\mathbb E}
\def\GG{\mc G}

\def\Ib{I_{\ms{b}}}

\def\R{\mathbb R}

\def\Var{\ms{Var}}
\def\Z{\mathbb Z}
\def\a{\alpha}
\def\ba{\,|\,}
\def\been{\begin{enumerate}}
\def\bee{\begin{example}}
\def\bel{\begin{lemma}}
\def\bepr{\begin{proposition}}
\def\bep{\begin{proof}}
\def\bet{\begin{theorem}}
\def\bec{\begin{corollary}}
\def\bede{\begin{definition}}
\def\ende{\end{definition}}

\def\bs{\boldsymbol}
\def\b{\beta}

\def\co{\colon}
\def\ddd{\mathfrak d}

\def\dist{\ms{dist}}

\def\d{{\rm d}}
\def\enen{\end{enumerate}}
\def\ene{\end{example}}
\def\enl{\end{lemma}}
\def\enpr{\end{proposition}}
\def\enp{\end{proof}}
\def\ent{\end{theorem}}
\def\enc{\end{corollary}}

\def\es{\varnothing}
\def\e{\varepsilon}
\def\ff{\infty}
\def\fr{\frac}

\def\g{\gamma}
\def\im{\item}

\def\ms{\mathsf}

\def\one{\mathbbmss{1}}
\def\pa{\partial}
\def\r{\rho}
\def\sm{\setminus}

\def\su{\subseteq}
\def\s{\sigma}
\def\tff{\uparrow\infty}

\def\ti{\times}

\def\vp{\varphi}
\def\wt{\widetilde}
\def\w{\wedge}

\def\su{\subseteq}

\def\bt{\beta_n}
\def\btt{\beta^{+}_{n}}
\def\FF{\mc F}
\def\wt{\widetilde}
\def\CC{\mc C}
\def\Om{\Omega}

\def\tbnp{\bar \b^+_n}

\usepackage{tikz}
\def\blue{}
\usetikzlibrary{hobby,calc,shapes.misc}
 \usepackage{mathabx}
\usepackage{pdfsync}
\usepackage{enumitem}
\usepackage[backref=page]{hyperref}
\usepackage[foot]{amsaddr}

\usepackage{scalerel}
\usepackage{amsmath,amsfonts,amssymb,amsthm,mathrsfs}
\usepackage{cancel,soul}

\renewcommand{\f}{  \mathsf f}

 \renewcommand{\E}{\mathsf E}
\newcommand{\equlaw}{\stackrel{( d)}{ = }}

 \newcommand{\C}{    \mathsf C}  
\renewcommand{\f}{  \mathsf f}
 \newcommand{\Hess}{\text{\rm{Hess}}}  
 
 \newcommand{\Ys}{Y^{\partial }}

 \renewcommand{\E}{\mathsf E}
\newcommand{\W}{\mathcal W}
\title{Functional Central Limit Theorem for topological functionals of Gaussian critical points}

\begin{document}

\author{Christian Hirsch}
\author{Raphaël Lachièze-Rey}
\address[Christian Hirsch]{Department of Mathematics\\ Aarhus University \\ Ny Munkegade, 118, 8000, Aarhus C, Denmark.}
\address[Raphaël Lachièze-Rey]{INRIA \\ Paris, France.}
\email{raphael.lachieze-rey@math.cnrs.fr}
\email{hirsch@math.au.dk}
\maketitle

\newcommand{\x}{ \mathcal{X} }

 {\bf Abstract}
 We consider Betti numbers of the excursion of a smooth Euclidean Gaussian field restricted to a rectangular window, in the asymptotics where the window grows to $ \mathbb R^d$. With motivations coming from Topological Data Analysis, we derive a functional Central Limit Theorem where the varying argument is the thresholding parameter, under assumptions of regularity and covariance decay for the field and its derivatives. We also show fixed- and multi-level CLTs coming from martingale based techniques inspired from the theory of geometric stabilisation, and limiting non-degenerate variance.\\


%
%
\section{Introduction}
\label{sec:int}
%
%

Stationary  Gaussian fields are a dominant model to represent  spatially homogeneous continuous data in several dimensions. In particular, the  {\it excursion sets}, or  {\it upper level sets}, obtained by thresholding at a given level,  are the topic of intensive research across several fields, and under many angles: image analysis, percolation, materials science, neurology, cosmology. Their geometric characteristics, such as the volume, perimeter, or topological indexes, have been the subject of several works on limit theorems. See for instance  \cite{AT07,SurveyBeliaev,GDFW} for theoretical tools, results, and case studies for smooth Gaussian fields.

Recently, many disciplines have adopted methods from the mathematical domain of topological data analysis (TDA) to analyse data exhibiting complex topological features. In this area, the  Betti numbers are the key tool, which, loosely speaking, describe the number of holes of a fixed dimension contained in a data set. In order to put TDA on a rigorous statistical foundation and to allow for the derivation of hypothesis tests, it is crucial to develop the theory of normal approximation for the Betti numbers. This has been successfully achieved for topology models based on the Poisson point process \cite{shirai,b1}. While point process models are interesting, when looking into applications,  random-field models are of central importance \cite{astro2,astro1,torquato}. However, when considering the asymptotics of persistent Betti numbers, then much less is known. One exception is a work on discretely-indexed Gaussian excursions, which is however restricted to the sparse regime \cite{thoppe}.

Central limit theorems (CLTs) for geometric functionals have been a topic of intensive research in the last 20 years. Here, the most thoroughly studied model is that of a Poisson point process in Euclidean space. For instance, we refer the reader to \cite{mal_shot,mal_stab,mehler,yukCLT} and the references therein.
Motivated by the emerging domain of TDA, we focus in this work on topological indexes of the fields.
There has been a lot of progress in the normal approximation of geometric functionals on Poisson processes, mostly based on the Stein-Malliavin technique developed by Nourdin and Peccati, see the general theory in the monograph  \cite{NPbook}, focused on Gaussian input. This method was exploited by Estrade \& L\'eon \cite{leon} to establish the first CLT for topological functionals of Gaussian excursions, precisely for the Euler characteristic.   When one wishes to compute chaotic decompositions, as in the latter work, the arguments critically rely on the local nature of Euler characteristic and therefore do not generalise to the more globally defined Betti numbers. Much more recently, \cite{bel} applied a general martingale technique to derive a CLT for the component count of Gaussian random fields under suitable moment assumptions, in the spirit of the theory of geometric stabilisation  \cite{yukCLT}. Later, McAuley  \cite{maca} established a similar CLT for geometric functionals of the unbounded connected component.

The results mentioned above in \cite{b1,shirai} concern a CLT at a fixed level. However, in applications, it is typically not at all clear what is the {\blue fixed level that should be chosen}.  Also, TDA is concerned with the evolution of the topology when a real parameter is varying, this role is played here by the level $ u$. Therefore, it is essential to have a functional CLT that allows to vary the level.  In our main result, we derive such a result under suitable moment conditions where the level is allowed to vary in sub-critical and super-critical regime of random-field percolation. We note that even in the point-process case such a process-level result is only known to hold for quasi-1D domains or under truncation of the Betti number \cite{cyl,svane}.  Here, {the particular }challenge in functional CLTs is the question of tightness for topological functionals.  One important ingredient in our arguments is the derivation of exponential tails for the diameter of bounded components coming from continuous percolation.

%
%
Besides tightness, we derive  CLTs at fixed levels for a large class of non-local topological functionals. We also discuss positivity of the limiting variance. While we can reuse some parts of  the strategies from the component-case considered in \cite{bel}, the general Betti numbers require a more careful geometric analysis.

%
%
%
The rest of the manuscript is organised as follows. In Section \ref{sec:gauss} we introduce the Gaussian fields and recall some useful properties  about them. In Section  \ref{sec:morse-topo}, we introduce the topological functionals and their fundamental properties. In Section \ref{sec:main-results} we state our main results. In Section \ref{sec:gaus}, we discuss key results on Gaussian random fields such as the white-noise decomposition. Section \ref{sec:topoa} contains important topological preliminaries that will be used in the proofs.
 In Section \ref{sec:fix_clt}, we prove the fixed-level CLT, and show the positivity of the limiting variance. Finally, the proof of the functional CLT is given in Section \ref{sec:prf-fclt}.

\section{Results}
\label{sec:gauss}

A stationary Gaussian field is a random function $F:\R^d\to \R $ with Gaussian finite-dimensional marginals $(F(t_1),\dots, F(t_d))$, and which is distributionally invariant under translations, i.e. $ F(t + \cdot )\equlaw F$ for $ t\in \R^d$. We assume furthermore that the field is centered with unit variance (i.e. $ F(x)\sim \mc N(0,1)$ for $ x\in \R^d$), in which case the law of $ F$ is characterised by the covariance function
\begin{align*}
\C(x)= \mathbf E(F(0)F(x)),x\in \R^d.
\end{align*}We state assumptions ensuring asymptotic independence and regularity of the field.

\begin{assumption} [Decay, regularity, and non-degeneracy assumptions]
\label{ass:gaussian-intro}~
	$ \C$ must be written $ \C = q\star q$ for some symmetric function $ q:\R^d\to \R$ of class $ \mc C^{l_0 + 2}$ for some $ l_0>2^{13}$ such that for each multi-index $ \a $ with $ | \a | \le 3$, $ | \partial ^{\a }q(x) | \le c(1 +  \|x\|)^{-\eta }$ for some $c>0, \eta >{255}d^2$.

	Also  we assume that $ \int q\ne 0${, \blue which is a condition we use to show the positivity of the limiting variance in our CLT.}
\end{assumption}

By  \cite[Section 1]{AW}, this assumption implies in particular that the Gaussian field a.s. has sample paths of class $ \mc C^{l_0 + 1}$. In practice, only class $ \mc C^3$ is necessary to properly assess quantitatively topological properties of the field, but   by  \cite{GassSte} this regularity yields that the number of critical points has locally finite moments of order $ l_0$ (Theorem \ref{thm:gass}), which is useful for using H\"older's inequality in several parts of the proof.

We study the topological properties of the excursion sets
\begin{align*}\E(u) = \E(u;F) = \{x:F(x)\ge u\},u\in \R,
\end{align*}
intersected through a rectangular window $ W\su\R^d$. We consider topological functionals applied to the connected components of $ \E(u)\cap W$. It is more convenient to work with components which do not touch $ \partial W$, and we shall assume that these components have a controllable size. For $ A\su \R^d,$ denote by $  \mathscr C(A)$ the set of connected components of $ A$. For bounded $Q\su A$, let $   \mathscr  C(A;Q)$ the  set of $ C\in \mathscr C(A)$ with $ C\cap Q\ne \es $, and let $ \mc C(A;Q)$ the union of the $ C\in \mathscr C(A;Q)$.  {Let furthermore $  \mathscr  C(A;Q;B)$ the components of $  \mathscr  C(A;Q)$ which do not touch $ B$, we will   investigate $  \mathscr  C(A;Q;\partial W)$ for some compact $ W$ containing $ Q$, hence such components are bounded.  {Such components at level $ u = 0$ are coloured in gray at Figure  \ref{fig:example} below}}%

The functional CLT established in this paper {\blue is valid }on an interval $ I$ away from the critical regime. To make this precise, we henceforth let $ u_c^s\ge 0$ be the threshold of sharp phase transition{, i.e. the infimum of values $u'$ such that for any $ a>0$ there is $ c_a<\infty $ such that for any $ u>u'$
\begin{align*}\mathbf P\big(\textrm{diam}(\mc C (\{F\ge u\},\{0\})\geqslant r) + \textrm{diam}(\mc C (\{F\ge -u\},\{0\}))\ge r\big)\le c_ar^{-a}.
\end{align*}
Remark that it is not automatic that $ u_c^s<\infty $, this condition is discussed below. Before that, we first state our assumption on the interval $I$.}

\begin{assumption}
\label{ass:perco}
 $ I\su (u_c^s,\infty )\cup (-\infty ,-u_c^s)$.
 \end{assumption}
 {\blue We impose Assumption \ref{ass:perco} primarily to control some long-range dependencies occurring in the proof. Note that in the component-count CLT from \cite{bel} no such percolation assumption is necessary. Therefore, we believe that it is plausible that the CLT for the Betti numbers continues to hold even if Assumption \ref{ass:perco} is violated.}

 Most available results actually give exponential decay. This assumption is discussed at  Section \ref{ss:perc}. {Let us discuss two widely used families of Gaussian fields.
 \been
\item For $ \nu >0$, the Mat\'ern covariance is of the form
\begin{align*}
 \C_\nu (x) = c_\nu \|x\|^\nu K_\nu (x), x\in \mathbb{R}^{ d},
\end{align*}
for some constant $ c_\nu >0$, where $K_\nu  $ is the modified Bessel function of the second type of index $ \nu $. It has spectral density of the form $ S_\nu (u) : = \hat \C_v(u)=  c_\nu '(1 + \|u\|^{ 2})^{ -\nu -d/2}$, see e.g.~\cite{stein}.  Henceforth, we will often use the close relation between moment and regularity when switching from the original density to the spectral density, see e.g.~\cite{fourier}.
Define $ q_\nu : =  (2\pi )^{ -d} \widehat{ \sqrt{ S_\nu }}$ so that $ \C_\nu  = q_\nu   \star q_\nu $.  For $ \nu $ sufficiently large, $ S_\nu $ has a moment of order $ 2l_0 + 4$ hence $ q_\nu $ is of class $ \mc C^{ l_0 + 2}$.  
The decay of the $ \partial _\a q$ at infinity is more subtle, as it is related to the regularity of $ \sqrt{S_\nu }$. For $ k\in \mathbb N$, $\|x\|^{ 2k}\partial _\a  q(x) $  is the Fourier transform of $s_{ k,\a }(u) =  D _{ 2k} (u^{   \a   } \sqrt{ S_\nu (u)})$ for some derivation operator $ D_{ 2k}$ of order $ 2k$. Hence, once again, for $ \nu $ sufficiently large, $ s_{ k,\a }$ is integrable and $ \|x\|^{ 2k}\partial _\a q(x)$ is bounded, and Assumption  \ref{ass:gaussian-intro} is satisfied. 
\item The second example is the Bargmann-Fock covariance, which is of the form
\begin{align*}
\C(x) = \exp(-\|x\|^{ 2}/2), x\in \mathbb{R}^{ d},
\end{align*}
and is defined in any dimension. Using similar steps as in the Mat\'ern case, the Bargmann-Fock field is seen to satisfy Assumption \ref{ass:gaussian-intro}. 
\enen
As discussed at Section  \ref{ss:perc}, these examples also satisfy Assumption \ref{ass:perco} for sufficiently high $ \nu $ and appropriate $ u_c^s$.}

\subsection{Topological functionals}
\label{sec:morse-topo}

For simplicity we assume here that  $W = W_n$ is as in \cite[Section 3.2]{bel} given by a growing sequence of boxes of bounded aspect ratio. More precisely, for $W = [a_1, b_1] \ti \cdots \ti [a_d, b_d]$,  let $\ms{asp}(W) := \min_i (b_i - a_i) / \max_i (b_i-a_i)$. Then, throughout the entire manuscript we always tacitly assume that $\bigcap_{n \ge 1}\bigcup_{m \ge n} W_m = \R^d$ and that $\inf_n \ms{asp}(W_n) > 0$.
 Our aim is to treat Betti numbers of Gaussian excursions, such as number of connected components, number of cavities, etc\dots We introduce a more general class of topological functionals adapted to Morse excursions.

We call
$
 \mathscr E^k$ the class of $ k$-dimensional compact manifolds $ A$ of $ \R^d$ (with boundary) that can be written as the level set $A =  \{f\ge u\}$ of a $ \mc C^k$ Morse function $ f:\R^d\to \R$, such that for two critical points $ x,y$ of $ f$, $ f(x)\ne f(y)$. The latter requirement could be avoided, but it is not restrictive for Gaussian fields and more convenient this way. Such a representation of $ A$ is called a \emph{Morse representation}.

 \begin{definition}
 \label{def:topo-func}
 Say that a function $ \b : \mathscr E^k\to \R$ is {\it topologically additive} if it admits the representation on connected components $ C$
\begin{align*}\b(A) =
\sum_{C\in  \mathscr C(A)}\b(C)
\end{align*}
and
 \begin{itemize}
\item $ \b(C) $ only depends on the isotopy class of $ C,$ where two sets $ C,C'$ are in the same isotopy class {if there is a continuous function $ \g :[0,1] \times \R^d\to \R^d,t\in [0,1]$ such that $ \g (0,x) = x$ and $ \g(1,C) = C'$, and for each $ t$, $\g (t,\cdot )$ is a homeomorphism.}
\item  {there is $ \kappa <\infty $} such that for any Morse representation $ C = \{f\ge u\}$,
 $ | \b(C) |  {\le \kappa N_{ f,u}} $, where $ N_{ f,u}$ is  the number of critical points of $ f$ over $ C$ above level $u$.
 \end{itemize}
 \end{definition}

	 Our main example will be the Betti numbers. The precise definition relies on the concept of homology, whose mathematical definition is beyond the scope of the present work. We refer the reader to \cite{milnor,hatcher}.%

	 \def\b{\beta}
\begin{definition} Betti numbers are additive topological functionals $ \b _k, 0\le k  < d$ where $ \b _k(A)$ counts the number of equivalence classes of $ k$-dimensional cycles of $ A$.  
\end{definition}
 For instance, a $ 0$-dimensional cycle is in the same class as a point, hence $ \b _0(A)$ counts the number of classes of points which are not topologically equivalent in $ A$, hence $ \b _0(A)$ is the number of bounded connected components of $ A.$

	 The first property from Definition \ref{def:topo-func} holds for Betti numbers since they are invariant by homotopy, see \cite[Theorem 2.10]{hatcher}. For the second property, note that if $u \le u'$ are two levels of critical points of $f$ such that there is no other critical point with level in $[u, u']$, then the relative homology groups $H_k\big(\{f \ge u\}, \{f \ge u'\}\big)$ are of rank 1 for at most one value of $k$ and 0 otherwise. Therefore, the second property
	 follows from the long exact sequence in homology. We refer the reader to \cite[Section I.5]{milnor} for details.

We will be looking at here a $ \mc C^{{\blue l_0 + 2}}$ Gaussian field $ F$ which level sets are not compact, but we only observe the connected components in the interior of a window $ W_n $, so that we have indeed
\begin{align*} \bigcup _{C\in \mathscr C(\E(u),W_n ,\partial W_n )}C = \{f\ge u\},
\end{align*}
 for some $ \mc C ^{{\blue l_0 + 2}}$ function $ f$ on $ W_n$ which coincides with $ F$ on a $ \e $-neighbourhood of the components for $ \e >0$ sufficiently small. It is well known that  excursions of smooth Gaussian fields   {of the form of Assumption  \ref{ass:gaussian-intro} are almost surely Morse, i.e. they a.s. do not have degenerate critical points (see  \cite{AT07}).}
We set
\begin{align}
\label{eq:b_n}\b _{n}(u) : = \b _n(u;F) :=\sum_{ C\in \mathscr  C(\E(u),W_n,\partial W_n)} \b(C).
\end{align}
Hence, we consider topologically additive functionals as defined above.

%
%
\def\b{\beta}
\def\qm{q_{\ms M}}
%
%

 \subsection{Fixed- and multi-level CLT and variance lower bounds }
 
\label{sec:main-results}
 \label{ss:sec}
 
We first state the  CLTs for  {\blue topological functionals}, then the FCLT in the variable $ u$, where we also assume that $ \#\{i: Q_{i}\neq \emptyset \}\sim \text{\rm Vol}(W_n)\sim n$  {where $ Q_{ i} =( i + [0,1)^{ d})\cap W_{ n},i\in \Z ^{ d}$}. {\blue Let $|A|$ denote the Lebesgue measure of a Borel set $A \subseteq \R^d$.}
 In this section, we state a CLT for {\blue topological functionals} at some fixed levels.     To ensure positivity of the limiting variance, we need an additional condition. To state this precisely, let   
 \begin{align}
	\label{eq:mue}
	\mu (u) :=\lim_{n\to \infty }  |W_n|^{-1}\mathbf E\big(\b_n(u)\big),
\end{align}
supposing that this limit exists. {\blue For example, this limit exists and is postive in the case of Betti numbers, see Remark \ref{rem:bet} below}.
Henceforth, we set $\g := \eta/d - 1/2$, which satisfies $ \gamma >{254}d$ under Assumption \ref{ass:gaussian-intro}. Define   
\begin{align*}
\tilde \beta _n^{}(u)=n^{-1/2}(\beta _n^{}(u)-\mathbf{E}(\beta _n^{}(u))).
\end{align*}
%
%
\bet[Multivariate CLT]
\label{thm:fix_clt}
Consider a Gaussian random field satisfying Assumption \ref{ass:gaussian-intro}.
  Furthermore, let {\blue $u_1, \dots, u_K\in I$ where $I$ is} as in Assumption \ref{ass:perco}.
Then,
 $$\{\wt{\b_n }(u_i)\}_{i \le K}\Rightarrow \mc N(0, \Sigma),$$
 where $\mc N(0, \Sigma)$ is a normal distribution with mean 0 and some covariance matrix $\Sigma$.  {If $\mu(u)$ exists and is not 0 for every $u \in \R$,   then the limiting variance is strictly positive, i.e., $\Sigma_{ii} > 0$.}
\ent

\begin{remark}
 {Our bound for $ l_0$ is likely  very conservative. We believe it can be improved substantially by optimizing further within proofs.}
\end{remark}

Theorem \ref{thm:fix_clt} is proven in Section \ref{sec:fix_clt} by invoking a general CLT for stabilizing functionals on Gaussian random fields from \cite[Theorem 1.2]{bel}. The latter result is modeled after a classical CLT for stabilizing functionals of a Poisson point process \cite[Theorem 3.1]{yukCLT}. Very recently, there also has been a general CLT implying the asymptotic normality of the volume, surface area and Euler characteristic of an unbounded component in the excursion set \cite{maca}.

Concerning the positivity of variance, we then proceed along the lines of \cite[Theorem 1.3]{bel}. However, a crucial ingredient in that proof is \cite[Lemma 3.13(1)]{bel}, which contains a delicate argument based on properties of the spectral measure to show that with positive probability at least some connected component of the excursion set intersects $[0, 1]^d$.

%
%
\subsection{Functional central limit theorem}
\label{ss:fclt}
After having established the CLT at   { fixed levels} in Theorem \ref{thm:fix_clt}, the next step is to prove a CLT where the functional is considered as a stochastic process in the level. To this end, we must ensure that the process $ u\mapsto \beta _{\blue n}(u)$ is in the CADLAG space:
\begin{lemma} {\blue Let $u_- \le u_+$ be real numbers such that there are no critical points in $W_n$ where the value of the field lies between $u_-$ and $u_+$. Then,} $ \beta _{\blue n}(u_-) = \beta _{\blue n}(u_+).$

\end{lemma}

Since, by Lemma \ref{lm:KR} below, the critical values over a compact window are almost surely locally finite, the process $ \beta _{\blue n}$ a.s.~jumps finitely many times on each compact. Therefore, we henceforth consider $\beta_n(u)$ as an element in the space of CADLAG functions equipped with the standard Skorokhod topology, see \cite[Section 12]{billingsley}. {Recall that $\tilde \beta _n^{}(u)=n^{-1/2}(\beta _n^{}(u)-\mathbf{E}(\beta _n^{}(u))).$}

\begin{theorem}[FCLT for  Betti numbers]
	\label{thm:fclt}
Consider a Gaussian random field satisfying Assumption \ref{ass:gaussian-intro}.	Let $I$ be an interval satisfying Assumption \ref{ass:perco}. Then, as $n \to \ff$, as a process, in the Skorokhod topology,
$${\wt{\b_n }(\cdot)} \Rightarrow Z,$$
where $Z$ is a centered Gaussian process  on the interval $I$.
\end{theorem}
{\blue Note that in general, we are not aware of more precise distributional results of the limiting process $Z$. However, in the one-dimensional case, the Betti number corresponds to the number of level crossings, which has been studied intensely in the literature \cite{AW}.
}

%
%
\subsection{Discussion of the percolation assumption}
\label{ss:perc}

This question is highly technical and outside the scope of the present paper,
many recent works deal with the percolation properties of Gaussian excursions, and in particular give conditions under which the percolation regime undergoes a {\it sharp phase transition}.
We briefly discuss what is known about the validity of Assumption \ref{ass:perco} on the exponential decay of the cluster diameters.
A more elementary question concerns the existence of a level above which with probability 1, there exists an unbounded connected component. That is, we define
$$u_c:= \sup\Big\{u \in \R: \mathbf P\big(\text{\rm diam}(\mc C(\E(u);[0, 1]^d)) = \ff\big) > 0\Big\}$$
as the standard critical level of percolation. 
The fact that for sufficiently low $ u$ there is an unbounded connected component under mild 	assumptions goes back to Stepanov and Molchanov \cite{SteMol}.
 Here are some more precise statements. We mainly give results in the supercritical regime, as it is in general easier to give bounds on the size of bounded components of the subcritical regime.
\been
\im Consider the planar case, i.e., $d=2$. Here, $u_c = u_c^s = 0$ by self-duality under very mild assumptions. Then, \cite[Theorem 1.7]{muir2} shows the sharpness of the phase transition under some correlation decay assumptions. In particular, Assumption \ref{ass:perco} is implied by Assumption \ref{ass:gaussian-intro} in the planar case, meaning it holds for every interval $ I$ not containing $ \{0\}$.

\im For general dimensions $d \ge 3$, the situation is more delicate. First, the typical situation is to have percolation of both phases at level $ 0$, hence $ u_c>0$; it has been proved to hold in \cite{DRRV} under stronger assumptions, and in general the value of $ u_c$ is not known, as in most non-symmetric percolation models. \cite[Theorem 1.2]{sev} shows the sharpness of the phase transition under mild hypotheses, meaning that Assumption \ref{ass:perco} holds under Assumption \ref{ass:gaussian-intro} for $ I\su [-u_c,u_c]^{c}$. However, the arguments need positive association (a.k.a. the FKG inequality) and therefore only apply for nonnegative covariance kernels. That is, we must additionally assume $\inf_x \C(x)\ge 0$. {\blue We also assume that $q$ is invariant under permutations and sign changes of the coordinate axes.}
\im {\blue Let $ B(x,R)$ be the ball centred in $ x \in \R^d$ with radius $ R > 0$. Let $E_u^{(x, R)}$ denote  the event that there exists a bounded excursion component intersecting both $B(x,R)$ and $\partial B(x,2R)$. }
Without positive association, for the components in the subcritical regime, we can apply \cite[Theorem 3.7]{Mui-sprinkled}: if the covariance decays monotonically at a rate faster than any polynomial, then we have the  threshold
\begin{align*}
	{\blue u_c' =\inf \big\{u\in \R:\liminf _{R\to \infty }\sup_{x\in \R^d}\mathbf P\big(E_u^{(x, R)}\big) = 0\big\}}
\end{align*}
that satisfies    ${\blue u_c'}\le u^s_c<\infty $, and for $ u>u^s_c,$
\begin{align*}
	\limsup_{r \to\infty}(\log r)^{-1} \log \mathbf P(\textrm{ diam}(\mc C(\{F\ge u\},\{x\}))>r) = - \infty.
\end{align*}
 It is expected that $ u_c^s= {\blue u_c'}$, but proved only in the planar case.
\im Finally, \cite[Theorem 1.2]{ms2} concerns again the sharpness of the phase transition for fields not satisfying the FKG condition in dimension $ d\ge 3$. The arguments here rely on a suitable finite-range decomposition of the considered random field. In particular, it is proved  that exponential decay occurs for some (non necessarily positive) covariances decaying polynomially with an arbitrary negative exponent.

\enen

To summarise:
\begin{theorem}[Muirhead, Rivera, Severo]Assume Assumption \ref{ass:gaussian-intro} holds. Then, Assumption \ref{ass:perco} holds in the following cases:\begin{itemize}
\item $ d = 2$ and $ 0\notin  \text{\rm{\color{black} closure}}(I).$  
\item $ d \ge 3$ and $ \C(x)\ge 0$ and $ I\su (-\infty ,-u_c^s)\cup (u_c^s,\infty )$. {\blue We also assume that $q$ is invariant under permutations and sign changes of the coordinate axes. }
\item $ d \ge 3$ and $ I\su (-\infty ,-u_c^s)$ for some $ u_c^s>u_c$ if $ \C(x)$ decays sufficiently fast.
\item $ d \ge 3$ and $ I\su (u_c^s,\infty )$ for some classes of covariances of ``finite range''
\end{itemize}
 \end{theorem}
{\blue We note that the final point probably also works for $ I\su (-\infty ,-u_c^s)$. The Bargmann-Fock and Mat\'ern examples from Section \ref{sec:gaus} below satisfy the case $ d = 2$. They also satisfy the case $ d\geqslant 3$, up to taking $ \nu $ sufficiently large in  the Mat\'ern family, as the decay can be made arbitrarily high. }

 \section{Properties of Gaussian fields}
 \label{sec:gaus}
 \subsection{White noise convolution}

Many assumptions are more conveniently stated through the spectral measure, defined as the unique probability measure $ \mu $ on $ \R^d$ such that
\begin{align*}\C(x) = \int e^{ixu}\mu (du), \qquad x\in \R^d.
\end{align*}

Let $\mc B(\R^d)$ denote the Borel $\sigma$-algebra on $\R^d$.
A Gaussian white noise is a random signed measure seen as a random field $\W :\mc B(\R^d)\to \R$ such that \begin{itemize}
\item $\W(A)\sim \mc N(0, | A | )$ for $A$ Borel bounded, where $  | A | $ is the Lebesgue measure of $ A$
\item $\W(A\cup B)= \W(A)+\W(B)$ a.s.~for every disjoint bounded Borel sets $A,B \su \R^d$
\item $\W(A)$ and $\W(B)$ are independent for every disjoint bounded Borel sets $A,B\su\R^d.$
\end{itemize}
See \cite[Section 1.4.3]{AT07} for an explicit construction.
It satisfies in particular for $f,g$ square integrable
\begin{align}
\label{eq:cov-WN}
\Cov \left(\int fd \W ,\int gd \W  \right)=\int fg.
\end{align}
Commenting on Assumption  \ref{ass:gaussian-intro}, it suffices to assume that $ \mu $ has a  smooth $ L^2$ density, denoted by $ \rho $, to ensure  $ \C= \hat \rho = q \star q$ with $ q= \widehat {\sqrt{\rho }}$.
Let $\W $ be a centred stationary Gaussian white noise on $\R^d$. Writing $\star$ for the classical convolution operator,  $ F$ admits the representation
\begin{align}
\label{eq:convol-repr}
F(x)\equlaw F(x;\W ):=q\star \W(x)
\end{align}
because
\begin{align*}
\C(x)=q \star { q}(x) =\int q(y)q(y+x)\d y =  \mathbf E(F(0)F(x))
\end{align*}
and the covariance uniquely determines the law of the Gaussian field.

 Not all Gaussian covariances can be written in this way, for instance the random planar wave model cannot as its spectral measure is singular  \cite[Section 2.1]{SurveyBeliaev}.
 {We sometimes work under the renormalisation assumption}
\begin{align*}
\Var(F(x))=\|q\|_{L^2}^{ 2}=1.
\end{align*}

\subsection{Non-degeneracy}
 \label{sec:ass}

  Let us recall the formulae linking derivatives of the field and the covariances \cite[(5.5.4)-(5.5.5)]{AT07}: if a covariance function $\C$ is $\mc C^{2k+1}$, $F$ is a.s. of class $\mc C^{k}$ and for natural integers $\a ,\eta ,\g ,\delta $ such that $\a +\eta \le k,\g +\delta \le k$,
coordinates $ 1\le i,j\le d,$
\begin{align}
\label{eq:deriv-cov}
\mathbf E\left(
\partial _i^{\a }
\partial _j^{\eta }F (t)\cdot \partial _i^{\g  }\partial _j^{\delta }F (s)\right)=
	\frac{ \partial^{\a +\eta +\g  +\delta } }{\partial {t_i}^{\a } \partial {t_j}^{\eta } \partial {s_i}^{\g  } \partial {s_j}^{\delta }}C (t-s),s,t\in \R^{\blue d}.
\end{align}
For instance, by symmetry,  {for some $ 0<\lambda _-<\lambda _{  + },$}
{for $ k\geqslant 2,$
\begin{align}\notag
\nabla C(0)=& 0\\
\label{eq:Cov-Hess}  \C(0)-\lambda _ + \|t\|^2 \le  \C(t)&\le \C(0)-\lambda _{  - }\|t\|^2.
\end{align}
Denote throughout all the paper $ \Hess _{ F}(x)$ the Hessian matrix of some $ \mathcal{C}^{ 2}$ field $ F$ at some point $ x.$
 In fact, $ \lambda _-,\lambda _{  + }$ are related to the (negative) eigenvalues of $\Hess_{ \C}(0)$, or equivalently to the maximum and minimum of  $  \textrm{Var}\left( \nabla F (0)\cdot u\right)$  over $ u\in  \mathbb  S  ^{ d-1}$.}

The following standard result states that the field's derivatives are not degenerate at disjoint locations under Assumption \ref{ass:gaussian-intro}{, it will be applied later on to the set of orderer multi-indexes of length $ \le 3$ in $ \{1,2\}.$} $$  \mc I = \{(1),(2),(1,1),(2,2),(1,2),(1,1,1),(1,1,2),(1,2,2),(2,2,2)\}.$$ 
\begin{proposition}
\label{prop:ND-strong}  {Let $  \mc I$ be any finite set of ordered multi-indices.}
Assume there is an open set in the support of the spectral density $ \rho $.
Let the Gaussian vector $ V(x) = (\partial _{\a }F(x);\a \in \mc I)\su \R^{ \#\mc I },x\in \R^d$ and for $ x,y\in \R^d$, $ V(x,y)\in \R^{2\# \mc I}$ the random vector obtained by concatenating $ V(x)$ and $ V(y).$
 Then
for all $ x\in \R^d$, the derivatives $ \partial ^{\a }F(x),\a \in \mc I$ form a non-degenerate Gaussian vector, i.e. $ \det( \Cov (V(x))>0$.
Also, for $ \delta >0,$
\begin{align*}
\inf_{ | x-y | >\delta } | \det( \Cov \left(V(x,y)\right)) | >0.
\end{align*}
\end{proposition}

 {For a $ \mathcal{C}^{ l}$ function $ G$ and $ 0\leqslant m\leqslant l$, denote by $ \partial ^{ \alpha }G, | \alpha  |  = m$ the derivatives for ordered multi-indices $ \alpha $ with $  | \alpha  |  = m$ to avoid repetitions of the same partial derivatives.}

 \begin{proof}
 The proof is based on the fact that for a Gaussian vector $ V = (V_1,\dots ,V_{m})$, $ \det(\Cov (V)) = 0$ iff there is a non-trivial linear relation
\begin{align*}\sum_{i = 1}^ma _iV_i = 0\textrm{ a.s.. }
\end{align*}

Let $(a _{\a })_{\a\in I },(b _{\a' })_{\a'\in I }$ finite collections of complex numbers indexed by multi-indices, and let $ x\in \R^d$. By recalling that $ \C = \hat \rho $, we have by
  \eqref{eq:deriv-cov}, for some polynomials $ P_1,P_2,Q: \mathbb C ^d\to \mathbb C ,$ for $ x\in \R^d \setminus \{0\}$,
\begin{align*} \textrm{Var}\left(
\sum_{\a }a_{\a }\partial ^{\a }F(0) + \sum_{\a' }b_{\a' }\partial ^{\a' }F(x)
\right) = \int_{ \mathbb C^d}[\underbrace{P_1(\lambda ) + P_2(\lambda ) + e^{i\lambda x}Q(\lambda )}_{R(\lambda )}]\rho (\lambda )d\lambda ,
\end{align*}
where resp. $ P_1,P_2$ are obtained when $ (b_{\a' })\equiv 0,$ resp. $ (a_{\a })\equiv 0$, and $ Q$ is the  {\it cross term}. This formula is valid for any $ L^2$ spectral density $ \rho $ and corresponding stationary Gaussian field $ F$, hence the right hand side is nonnegative for any $ \rho $. It implies that $ R(\lambda )\in \R_{ + }$, and similarly $ P_1,P_2$ are nonnegative.

	 Assume now that for some $ \rho $ having a nonempty open set $ O$ in its support, this quantity vanishes, which equivalently means that there exist deterministic complex $(a_\a)$, $(b_{\a'})$ such that almost surely,
	 $$ \sum_{\a }a_{\a }F(0) + \sum_{\a' }b_{\a' }\partial ^{\a' }F(x) = 0.$$
  It means that $ R(\lambda ) = 0$ over $ O$, and as an analytic function, it means it vanishes on $ \mathbb C ^d$. Hence, $ e^{i\lambda x}Q(\lambda )$ should have a finite expansion with $ x\ne 0$, which means that $ Q = 0$, and then $ P_1 + P_2 = 0$ as well. Since they are nonnegative, all coefficients of $ P_1,P_2$ are zero, which easily implies with \eqref{eq:deriv-cov} that the coefficients $ a_{\a }$ and $ b_{\a' }$ are all $ 0$.
 We hence proved by contradiction that $ \det( \Cov (V(x,y)))>0$ for $ x\ne y$. Looking at the first $ | \mc I | $ coordinates of $ V(x,y)$, it means that $ V(x)$ is non-degenerate, which proves the first statement.

 For the second statement, note that the map $(x,y)\mapsto \det( \Cov (V(x,y)))$ is continuous and does not vanish, hence for $0<\delta < K<\infty $,
\begin{align*}\inf_{ \delta <| x-y | <K} | \det(\Cov(V(x,y))) | >0.
\end{align*}
Let us finally prove by contradiction that the infimum over distant $x,y $ is non-zero as well. If it is zero, it means by stationarity that for some sequence $ x_n\to \infty ,$
\begin{align*}\det \underbrace{\Cov (V(0,x_n))
 }_{ = :\Gamma _{x_n}}\xrightarrow[x_n\to \infty ]{}\;0.
\end{align*}
Denoting by $ U_{x_n}$ a unit vector associated to the smallest eigenvalue of $ \Gamma _{x_n}$, we have $ \Gamma _{x_n}U_{x_n}\to 0$ in $ \R^d$.
By compactness of $ \mathbb S ^{d-1}$, it means we can find coefficients $ a_{\a },b_{\a' }$ which constitute the limit of a subsequence of $ U_{x_n}$ in $ \mathbb S ^{d-1}$ and such that
\begin{align*}\sum_{\a }a_{\a }\partial ^{\a }F(0) + \sum_{\a' }b_{\a' }\partial ^{\a' }F(x_n)\to 0.
\end{align*}
Since by stationarity both terms of the left hand side have a constant positive variance, it means their correlation goes to $ 1.$
This is in contradiction with  Assumption  \ref{ass:gaussian-intro}, which implies that for each $ \a ,\a' \in  \mc I,$ we have  $ \partial ^{\a  }\partial ^{\a' }\C(x_n)\to 0$ as $ x_n\to \infty .$

\end{proof}

\subsection{Concentration}

Finally, it is a standard  result in Gaussian processes that the field and its derivatives concentrate well.  {Define for $ A\su \R^{ d}$ and any $ G:A\to \R^{ q}$ for some $ q\geqslant 1$
\begin{align*}
 \|G\|_A = \sup _{ x\in A}\|G(x) \|.
\end{align*}}

\begin{proposition} \label{prop:concentr}  {  Let $G$ be a Gaussian field over some compact $A\su  \R^{ d}, \sigma _A^2: =  \sup_{ x\in A} \textrm{Var}\left(G(x)\right)$ and $ \sigma ^2\geqslant \sigma _A^2 $. Assume  for some $ \lambda <\infty $
\begin{align}
\label{eq:scale-GFs-concentration}
 \sigma ^{ -2}\mathbf E [\|G(x)-G(y)\|^2]\le   \lambda  ^2\|x-y\|^2, \qquad x,y\in A.
\end{align}
 There is  $ c >0$ depending only on $d, \lambda  $ such that for  $ t\geqslant 0$
\begin{align}
\label{eq:altern-concentr}\mathbf P(\|G\|_A\ge t)\le  c  (1 +  \mathsf{ diam }(A)^{  {d}})\exp(-t^2/2\sigma^2).
\end{align}}

\end{proposition}

{
\begin{proof}It suffices consider the case $ \sigma  = 1$, which we assume in the proof. Also, covering $ A$ with subsets $ A_{ i},i = 1,\dots ,N$ with diameter $ 1$, the union bound yields
\begin{align*}
 \mathbf P (\|G\|_{ A}>t)\leqslant \sum_{i}\mathbf P (\|G\|_{ A_{ i}}>t),
\end{align*}
and we can choose the $ A_{ i}$ so that $ N\leqslant \kappa _{ d}( 1 +  { \rm diam}(A))^{ d}$,
so it suffices to write the proof for $ A$ with diameter $ 1.$
  Borell-TIS inequality ( \cite[Thm. 2.1.1]{AT07}) yields for $ t\geqslant 0$
\begin{align*}
 \mathbf P\big( \|G\|_A >\mathbf E( \|G\|_A )+t\big)\le \exp(-t^2/2\sigma _A^2)\le \exp(-t^2/2),
\end{align*}recalling $ \sigma  = 1\geqslant \sigma _{ A}.$
Define the metric $   \mathsf d$ induced by $ G$\begin{align*}
  \mathsf d(x,y) = \sqrt{\mathbf E \|G(x)-G(y)\|^2},x,y\in A.
\end{align*}
 To estimate $ \mathbf E  \|G\|_A $, we use \cite[Th.1.3.3]{AT07}:   for some universal constant $ K,$
\begin{align*}
	\mathbf E (\|G\|_A)\le K\int_0^{  \text{\rm{diam}}(A)/2}\sqrt{\ln(N(\e;A ))}d\e
\end{align*}
  where $N(\e ;A)$ is the minimal number of $   \mathsf d$-balls with radius $ \e $  require to completely cover $ A$.
By Assumption  \eqref{eq:scale-GFs-concentration}, $ B(x,\e )\su B_{   \mathsf d}(x,{\lambda  }\e )$. Basic geometric considerations  yield $ \kappa_{ d} >0$   such that    $ N(\e;A )\le \kappa_{ d}( 1 + \text{\rm{diam}}(A)^{ d}(\lambda  \e) ^{ -d}  )$. Hence, given that $  { \rm diam}(A) \leqslant  1,$ $  \mathbf E \|G\|_A$ has a uniform upper bound depending on $ \lambda $ and $ d.$
\end{proof}
}

 {\begin{remark}
For a fixed Gaussian field $ F$ of class $ \mc C^{ 4}$, by \eqref{eq:Cov-Hess}, there is $ \lambda <\infty $ such that the result above applies to $ F$ and its derivatives on each compact. Throughout the article, we fix some field $ F$ and some $ \lambda $ adapted to $ F$ and its partial derivatives up to order $ 2$.
 \end{remark}}

\section{Topological analysis}
\label{sec:topoa}

\subsection{Morse representation}
Following the theory of Morse functions \cite{milnor}, the central objects of investigation in our work are the random critical points in the compact sampling window $W_n \su \R^d$ of volume $n$. More precisely, for a bounded Borel set $W \su \R^d$, given $ F:W\to \R$ smooth, we let
$$Y(W\times I;F) : = Y(W\times I): = \big\{\big(x, F(x)\big) \in W\times I \co \nabla F(x) = 0\big\}$$
be the marked point process of critical points of the field $F$.

The index $n$ indicates the restriction to some large rectangular window $W_n\su \R^d$
\begin{align*}
Y_n=Y(W_n\times \cdot ).
\end{align*}

Furthermore, let $Q_i$ be the intersection of $ W_n$ with $   \mathsf Q_i = i + [0,1)^d$.
For a window $ W =  \times_{ i = 1}^{ d}[a_i,b_i]$,   {for $ P\su \{1,\dots ,d\}$ with cardinality $ p$, the points of $ W$ whose $ i$-th coordinate is fixed either to  $ a_i$ or to $ b_i$ for  $i\notin P $ is called a facet of dimension $ p$. }
  In this context, call {\it stratified critical point}   any $ x$ belonging to a facet $ \f$ of $ W$ such that
$ F(x) = 0,\nabla _{\f}F(x) = 0$, where $ \nabla _{\f}F(x) = ( \partial _{u_i}F(x))_{ i\in I} $, for some basis $ u_i$ spanning the subspace containing $ \f$. Denote by $ \Ys(W\times I;F) = \Ys(W\times I)$ the corresponding process of stratified critical points, note that it contains $ Y(W\times I)$ (obtained for $ P = \{1,\dots ,d\}$).

The following result ensures that we have finiteness of the moments of the measure $ Y ^{ \partial }$:
\begin{theorem}[Gass, Stecconi \cite{GassSte}, Theorem 1.2]
\label{thm:gass}
 {Let $l\in \mathbb N, G$ a real  Gaussian field  of class $ \mc C^{l+1}$ a.s. on some bounded Borel set $ A\su \R^d$.   Assume that for $ x\in A,$ $ (\partial _{ \a }G(x))_{  | \a  | \le l}$ is a non-degenerate Gaussian vector}. {Then    the number   of critical points of $ G$ over $ A$ has a finite moment of order $ l.$}

 \end{theorem}

 {This result applies to $ F$ thanks to Proposition  \ref{prop:ND-strong} for $ l$ such that $ F$ is of class $ \mc C^{ l + 1}$.}
 The following lemma ensures some sort of topological stability of such a manifold $A = \{f\ge u\}$ perturbed by a $ \mc C^k$ real function $ \Delta:B\to \R $ where $ B$ is a neighbourhood of $ A$. Call {\it quasi critical point}  of  $(f,\Delta )$ at level $ u$ a couple $ (t,x)\in [0,1]\times  \R^d$ such that $(f + t\Delta )(x) = u, \nabla (f + t\Delta )(x) = 0$. In a stratified window $ W$, call more generally {\it stratified quasi critical point} a couple $( t,x)$ with $ x\in \f$ for some facet $ \f$ of $ W$ such that $ (f + t\Delta )(x) = u$ and $ \nabla _{\f}(f +t \Delta )(x) = 0.$
We typically consider the components $ \mathscr C(\{F\ge u\},W,\partial W)$ of the excursion set interior to $ W$.

 \begin{lemma}[Fundamental lemma of stratified Morse theory, Lemma B.1 of  \cite{bel}]
 \label{lm:Morse}
 Assume $ (f,\Delta )$ does not have stratified quasi critical point at level $ u$ on some stratified window $ W$. Then there is a one-to-one mapping $ \zeta $ between $ \mathscr C(\{f\ge u\},W,\partial W)$ and $ \mathscr C(\{f + \Delta \ge u\},W,\partial W)$, and for any $ C\in \mathscr C(\{f\ge u\},W,\partial W)$, $ C$ and $ \zeta (C)$ are isotopic.
 \end{lemma}

 This lemma implies that contributions in a topological functional over the manifold are determined by the (stratified) critical points of the function.
 Let us give an alternative representation of $ \b_n(u,F) $ obtained by scanning the levels and account carefully for critical points on $ \partial W_n$:

 \begin{proposition}
 \label{prop:decompos-b-n}
 For $ x\in \E(u)$, let $ C_x  $ be the bounded connected component  of $ \E(u)$ containing $ x$, or the empty set if $ x$ is not in a bounded component. Then, with Definition   \ref{def:topo-func}
 \label{repr:scan}
\begin{align*}\b _n(u;F) = \sum_{(x,v)\in \Ys(W_n\times [u,\infty ))}\delta (x,v,C_x),
\end{align*}
where $ \delta (x,v,C_x)$ satisfies  {for some $ \kappa <\infty $}
\begin{align*} | \delta (x,v,C_x) | \le  {\kappa } \#\Ys(C_x\times [u,\infty ))
\end{align*}
and $ \delta (x,v,C_x)$ only depends on the isotropy class of $ C_x$ (hence is not modified upon a $ \Delta $-perturbation without stratified quasi critical point by Lemma \ref{lm:Morse}).
 \end{proposition}

 \begin{proof}
Recall that $\b _n(u;F) = \sum_ C\b(C)$
 and, exploiting the fact that there are a.s.~finitely many stratified critical points on each facet  (Lemma \ref{lm:KR}), all with disjoint values, define for $ (x,v)\in \Ys(C\times [u,\infty ))$
\begin{align*}\delta (x,v,C) = \b _n(v^{ + };F)-\b _n(v^{-};F)
\end{align*}
 where exponents $  + $ and $ -$ denote respectively lower and upper limits in $ v.$
 \end{proof}

\begin{remark}
 {
Even though we are only counting internal components, it is necessary to make a summation also on critical points on facets as otherwise an internal critical point belonging to a component hitting the boundary might be wrongly counted as adding a new internal component, see for example around the top left corner in Figure  \ref{fig:example}.}
\end{remark}

  We then define
the positive and negative parts:
\begin{align*}\b _n^{ + }(u;F) = \sum_{(x,v )}[\delta (x,v,C_x)]_{ + }\\ \b _n^{ - }(u;F) = \sum_{(x,v)}[\delta (x,v,C_x)]_{ - }.\\
\end{align*}
 The representation $ \b _n = \b _{ n}^{  + }-\b _{ n}^-$  with two non-decreasing functionals will be useful in Section \ref{sec:prf-fclt} when assessing the uniform tightness of the functional $ u\mapsto \b _n(u;F).$
Let us enumerate some situations for a triple $ (x,v,C)$ where $ \delta (x,v,C)\ne 0$. It is important to remember that this value is allocated in the scanning process and only depends on levels $ w\ge v$, and the value will not change when crossing critical points below {\blue $v$. The following three cases are illustrated also in Figure   \ref{fig:example}.}
 \begin{enumerate}
\item $x$ only involves internal components, such as when it is the  {uppest or lowest} point of a component or   of a hole,
or the merging of two internal components $ C,C'\in \mathscr C(\{F\ge v\};W_n,\partial W_n)$, in the latter case $ \delta (x,v,C) = \b(C\cup C')-\b(C)-\b(C')$.

\item $ x$ is a stratified critical point on $ \partial W$ where a component $ C$ touches $ \partial W_n$ at level $ v$, hence $ \delta (x,v,C) = -\b(C)$,
\item $ x$ is an internal critical point which merges an internal component $ C$ with an external one, in which case $ \delta (x,v,C) = -\b(C)$.  
\end{enumerate}

\begin{figure}[h]
  \centering

\begin{tikzpicture}[scale = 1]
  \tikzset{
    win/.style={draw=black, line width=1.2pt, rounded corners=6pt},
    internal/.style={fill=gray!35, smooth cycle, tension=0.95},
    boundary/.style={fill=blue!25, smooth cycle, tension=0.95}
  }

  \draw[thick] (0,0) rectangle (15,6);

  \draw[thick]
    plot [smooth cycle, tension=0.7]
      coordinates {(2,1.2) (5,1.4) (6.5,2.5) (5.4,4.5) (3.2,4.0) (2.2,2.4)};

  \node[cross out, draw, very thick, minimum size=6pt, inner sep=0pt] (xmark) at (4,3) {};
  \node[anchor=west] at ($(xmark.east)+(0.15,0)$) {$ (2)1$};

   \node[cross out, draw, very thick, minimum size=6pt, inner sep=0pt] (xmark) at (7.8,3.4) {};
  \node[anchor=west] at ($(xmark.east)+(0.15,0)$) {$ (1) 1$};

   \node[cross out, draw, very thick, minimum size=6pt, inner sep=0pt] (xmark) at (9,3.4) {};
  \node[anchor=west] at ($(xmark.east)+(0.15,0)$) {$(1)  1$};
   
  \draw[thick] (8,3) circle (0.8) (9.6,3) circle (0.8) (8.8,3) ;

  \draw[thick] (10.2,4.2) .. controls (9.4,5.2) and (11.3,6) .. (12,6) .. controls (12.7,6) and (14.0,5.2) .. (13.2,4.0) .. controls (12.6,3.6) and (11.2,3.6) .. (10.2,4.2) -- cycle;
  
  \node[cross out, draw, very thick, minimum size=6pt, inner sep=0pt, label=above:{$(0) -1$}] at (12,6) {};

  \node[cross out, draw, very thick, minimum size=6pt, inner sep=0pt, label=above:{$(2) 1$}] at (12,5) {};
   
  

\def\x{1.8}
\def\y{2}
\draw[boundary,thick] (5.4-\x-\y,4.5)
  .. controls (4.8-\x-\y,4.5) and (4.6-\x-\y,4.1) .. (4.7-\x-\y,3.7)
  .. controls (4.9-\x-\y,3.0) and (5.9-\x-\y,3.0) .. (6.1-\x-\y,3.7)
  .. controls (6.2-\x-\y,4.1) and (6.0-\x-\y,4.5) .. (5.4-\x-\y,4.5) -- cycle;

  \draw[boundary,thick] (1.9-\x,6) arc[start angle=180, end angle=360, radius=1.5];

  \node[cross out, draw, very thick, minimum size=6pt, inner sep=0pt, label=above:{$(1) -1$}] at (1.5,6) {};

  \node[cross out, draw, very thick, minimum size=6pt, inner sep=0pt, label=right:{$(2) 1$}] at (1.5,5.2) {};

  \node[cross out, draw, very thick, minimum size=6pt, inner sep=0pt, label=above:{$(2) 1$}] at (1.5,3.5) {};

  \node[cross out, draw, very thick, minimum size=6pt, inner sep=0pt, label=right:{$(0) -1$}] at (1.5,4.5) {};

  \draw[internal]
    plot [smooth cycle, tension=0.7]
      coordinates {(2,1.2) (5,1.4) (6.5,2.5) (5.4,4.5) (3.2,4.0) (2.2,2.4)};

	\draw[boundary] (10.2,4.2) .. controls (9.4,5.2) and (11.3,6) .. (12,6) .. controls (12.7,6) and (14.0,5.2) .. (13.2,4.0) .. controls (12.6,3.6) and (11.2,3.6) .. (10.2,4.2) -- cycle;
  
  \draw[internal] (8,3) circle (0.8) (9.6,3) circle (0.8) (8.8,3) node[cross out, draw, very thick, minimum size=6pt, inner sep=0pt, label=right:{$(0)-1$}] {};

  \node[cross out, draw, very thick, minimum size=6pt, inner sep=0pt] (xmark) at (4,3) {};
  \node[anchor=west] at ($(xmark.east)+(0.15,0)$) {$ (2)1$};
  
  
   \node[cross out, draw, very thick, minimum size=6pt, inner sep=0pt] (xmark) at (7.8,3.4) {};
  \node[anchor=west] at ($(xmark.east)+(0.15,0)$) {$ (1) 1$};

   \node[cross out, draw, very thick, minimum size=6pt, inner sep=0pt] (xmark) at (9,3.4) {};
  \node[anchor=west] at ($(xmark.east)+(0.15,0)$) {$(1)  1$};

  \node[cross out, draw, very thick, minimum size=6pt, inner sep=0pt, label=above:{$(0) -1$}] at (12,6) {};

  

\draw[fill=white, double=gray!35, line width=0.7pt, double distance=8mm, line cap=round]
(13,1.8) ++(19.47:1.2cm) arc[start angle=19.47, delta angle=360-2*19.47, radius=1.2cm];
  
  \node[cross out, draw, very thick, minimum size=6pt, inner sep=0pt, label=right:{$(1) 1$}] at (12,1.8) {};

  \node[cross out, draw, very thick, minimum size=6pt, inner sep=0pt, label=right:{$(0) -1$}] at (14.1,1.8) {};

 \begin{scope}[shift={(2.0,-1.2)}]
    \path[internal] plot coordinates {(0.0,0.7) (0.0,0.3) (0.4,0.3) (0.4,0.7)};
    \node[anchor=west] at (0.35,0.5) {\scriptsize Internal component};
    \path[boundary] plot coordinates {(3.0,0.7) (3.0,0.3) (3.4,0.3) (3.4,0.7)};
    \node[anchor=west] at (3.55,0.5) {\scriptsize Intersecting boundary};

  \node[cross out, draw, very thick, minimum size=6pt, inner sep=0pt] (xmark) at (10,6) {};
  \node[anchor=west] at ($(xmark.east)+(0.15,0)$) {$(1) 1$};
  \end{scope}
\end{tikzpicture}

  \caption{{\blue Fictitious example thresholded at level $ 0$, each cross marks a critical point or a quasi-critical point  $ x$ at level $ v\geqslant 0$, the mark indicates $ (v) ,\delta(x,v,C) $ where $ C$ is the component of $ x$.}}
  \label{fig:example}
\end{figure}
Here is an alternative strategy of assigning weights to critical points.
\begin{definition}[Reference point]

For a compact connected component $ C$ of $ \mathscr 	C(A;W;\partial W)$, call $ x(C;F) = x(C)\in W \setminus \partial W$ the critical point of $ C$ lowest with respect to the lexicographic order, called {\it reference point of $ C$}, and define
\begin{align*}\delta^{\text{\rm{ref}}} (x,v,C) = \b(C){\mathbf 1}\left[ x = x(C) \right]
\end{align*}for $ (x,v) \in Y(C\times [u,\infty )).$ We have indeed
\begin{align*}
\b _n(u,F) =\sum_{(x,v)\in Y(W_n\times [u,\infty ))}\delta ^{\text{\rm{ref}}}(x,v,C_x)  {1_{ C\cap \partial W_{ n} = \emptyset }}.
\end{align*}

 \end{definition}

It will be apparent in the proof of Lemma \ref{lm:proba-topo} that this representation is easier to handle when trying to evaluate the probability that the topology is not modified upon the perturbation by some field $ \Delta $, i.e. $ \b_n (u;F) = \b_n (u;F + \Delta )$. The reason is that it is easier to bound the probability that two critical points exchange lexicographic order during the perturbation than to bound the probability that they exchange value, i.e. that one becomes lower than the other.

\subsection{Topological perturbation}
\label{sec:perturb}
For $B\su \R^d$, define by $\W^{(B)}$ an independent resampling of the Gaussian white noise $\W $ in $B$, i.e.

\begin{align*}
\W^{(B)}(A)= \W(A\setminus B)+\W'(A\cap B)
\end{align*}
where $\W'$ is a white noise independent of $\W$ with the same law. We will only consider countably many such resamplings, so we can assume all $ \W^{(B)}$ are independent.
Define as in \eqref{eq:convol-repr}
\begin{align*}
F^{(B)}(x):=&q\star \W^{(B)}(x)\\
\Delta _ B:=&F-F^{(B)}.
\end{align*}

Remark that $ F^{(B)}$ has the same distribution as $ F$ because $ \W$ and $ \W^{(B)}$ have the same law. Also, $ \Delta _ B$ should be small far away from $ B.$
To quantify this, Assumption \ref{ass:gaussian-intro} yields that it  is possible to apply the following result to the field $ F $ and its derivatives.

\begin{proposition}
\label{prop:bd-Delta}Let $ A\su \R^d$  {convex}, $ B\su \R^{ d},r = 1 + d(A,B).$
For $ | \a | \le 3$,  { $ \partial _{ \a }\Delta _B$ satisfies \eqref{eq:scale-GFs-concentration}  with $ \sigma ^2 $ of the form $c r^{ d-2\eta }$, and $ \lambda $ depending on $ q$, hence} there is finite $ c>0$ depending on $ \a ,q,d$ such that for $ t\geqslant 0$
\begin{align*}
\mathbf P(\|\partial _{\a }\Delta _ B\|_A>t)\le c(1 +  \text{\rm{diam}}(A)^d)\exp\left(-\frac{ct^2}{(1+d(A,B))^{-2\eta +d }}\right).
\end{align*}

\end{proposition}

\begin{proof}
 { We have  with  \eqref{eq:cov-WN} and Assumption \ref{ass:gaussian-intro}, for $ x\in A,$}
\begin{align*}
  \Var(\Delta_ B (x) )=\Var(q\star(1_ B(\W-\W'))) = & 2 \int_ B | q(x-y) |^2 dy  \\
 &\le c'\int_{ B(0,r)^c}(1+\|x-y\|)^{-2\eta }dy\le c''r^{-2\eta +d},
\end{align*}
and  { \eqref{eq:scale-GFs-concentration} is satisfied. Indeed,} for $ x,y\in A,$
\begin{align*}
\sigma ^{ -2}\mathbf E [  \Delta _B(x)-   \Delta _B(y)]^2 = & \sigma ^{ -2}\int_B(q(x-z)-q(y-z))^2dz\\
 \le &\sigma  ^{ -2} \int_ B\|x-y\|^2\sup_{ t\in A} \| \nabla q(t-z)\|^2dz\\
 \le &c\sigma  ^{ -2} \|x-y\|^2\int_ B(1 + d(z,A))^{ -2\eta }dz\\
 \le &c'\sigma ^{ -2}\|x-y\|^2\int_{r}^\infty (1 + s)^{ -2\eta }s^{ d-1}ds\\
 \le &c '\|x-y\|^2.
\end{align*}
and $ c'$ depends on $ q,d$. Then the conclusion comes from Proposition \ref{prop:concentr}. Since the derivatives of $ q$ satisfy the same hypothesis, the proof works exactly the same with $ \partial _\a \Delta _B$ instead of $ \Delta _B$ (and $ \partial _\a q$ instead of $ q$).
 \end{proof}

Henceforth, we only consider $B:=H_{i,j}$  the open half-space of points closer from some $j\in \mathbb Z ^d$ than from some $i\in \mathbb Z ^d$, in which case use the shorthand notation
\begin{align*}
F^{(i,j)} = F^{(H_{i,j})};\;\Delta _{{i,j}}=\Delta _{H_{i,j}}
\end{align*}
and remark that $\Delta _{i,j}$ is independent from $\Delta _{j,i}$ because $ H_{i,j}\cap H_{j,i} $ has negligible intersection.

In the remainder of this section and the entire paper, it will be essential to be able to bound the expected number of critical points whose value is contained in a certain interval $I$. This will be done with the classical Kac-Rice formula.

\begin{lemma}[Kac-Rice]  
\label{lm:KR}
	Let $ m\ge 1,  Q\su \R^m$ compact and $ G: Q\to \R$ a $ \mc C^3$ smooth {stationary} centred Gaussian field. Let $ W_x = \big(\nabla G(x),\partial _{i,j}G(x),i\le j\big).$ Assume
\begin{align}
\label{ass:ND-AT}
  {(G(x),W_{ x})}\textrm{ is a non-degenerate Gaussian vector  {with density bounded by some $ c _{ G}<\infty $}. }
\end{align}
Then for $ I\su \R$
\begin{align*}\mathbf E(\#\{x\in Q: \nabla G(x) = 0, G(x)\in I\})\le c  | Q | | I |
\end{align*}
where $ c<\infty $ depends  on   the law of $ G$. 
\end{lemma}

\begin{proof}
Assumption \eqref{ass:ND-AT} yields by  Corollary 11.2.2 in \cite{AT07} that

\begin{align*}\mathbf E(\#Y(Q\times I)) = \int_{Q }{\mathbf E( |\det \Hess_{G}(x) | \mathbf1_{\{G(x)\in I\}} | \nabla G(x) = 0 )}  {f(x)}dx \end{align*}
 {where $ f(x)$ is the  density of  $ \nabla G(x)$ in $ 0$, uniformly bounded by assumption.}
Denote by $ \Lambda _{v}$ the (Gaussian) conditional distribution of $ G(x)$ given $W_x = v$. By basic results on conditional Gaussian vectors, its law is of the form $ \Lambda _{v} \sim \mc N(m(v, { x}),\sigma _x)$, in particular the variance $ \sigma _{ x}$ does not depend on $ v$ and is bounded from below by some $ \sigma >0$,  hence the conditional density is bounded by $ c _{ G}'<\infty .$ Then
\begin{align*}
\mathbf E( |\det \Hess_{G}(x) | \mathbf1_{\{G(x)\in I\}} | \nabla G(x) = 0 ) =& \mathbf E(\mathbf P(G(x)\in I | W_x) |\det \Hess_{G}(x) || \nabla G(x) = 0)\\
 =& \mathbf E(\mathbf P( \mc N(m(W_x, x)\in  I  )) |\det \Hess_{G}(x) || \nabla G(x) = 0)\\
 \le &c_{ G}'{ | I | } \mathbf E( |\det \Hess_{G}(x) | | \nabla G(x) = 0),
\end{align*}
and these quantities do not depend on $ x$ by stationarity.
As before, $ \Hess_{ G}(0)$ conditional to $ \nabla G(0) = 0$ is a non-degenerate Gaussian vector, hence it has finite moments of all order, and $ \mathbf E( |\det \Hess_{G}(0) | | \nabla G(0) = 0)<\infty $, as asserted.
\end{proof}

\subsection{Topological lemma}
{Fix a window $ W_n$ and   recall that } $  \mathsf Q_i  { = i + [0,1)^{ d}, i\in \mathbb Z ^{ d}}$, and  $ Q_i: =  \mathsf Q_i \cap W_n.$ Let $ \b $ a topologically additive functional as in Definition \ref{def:topo-func}.
We have the decomposition
$\b(u;W_n) = \sum_{i:Q_i\ne \es }\b _{[i]}(u;W_n)
$
where
\begin{align*}\b _{[i]}(u;W_n)&: = \b _{[i]}:= \sum_{(x,v)\in Y(Q_i,[u,\infty ))}
\delta^{\text{\rm{ref}}} (x,v,C_x)\\
&= \sum_{C\in \mathscr C(\E(u),Q_i,\partial W_n)}\b(C)\mathbf1_{\{{ x(C;F)\in Q_i }\}} .
\end{align*}
 We use implicitly that a.s.~no critical point is on the boundary of a $ Q_i$, formally proved with Lemma \ref{lm:KR} applied with the Lebesgue-zero set $ \cup_i \partial Q_i$.

In this section, we consider the effect above $ Q_i$ of a perturbation applied to the field. More precisely, the white noise is resampled far away in $ H_{i,j}$ (points closer from $ j$ than $ i$), and we denote by $ \tilde \b _{[i],j}$ the value of $ \b _{[i]}$ after perturbation, i.e. for the field $ F + \Delta _{i,j}$:
\begin{align*}\tilde \b _{[i],j}{ (u;W_n): = \tilde \b _{[i]}^j:} = \sum_{C\in  \mathscr C(\{F + \Delta _{i,j}\ge u\},Q_i,\partial W_n)}\b (C)\mathbf1_{\{{ x(C,F + \Delta _{i,j})\in Q_i }\}}.\end{align*}

 {We extend these definitions to the interval $ I = [u_- , u_+]$:
\begin{align*}
\b _{[i]}(I) = \b _{[i]}(u_+;W_n)-\b _{[i]}(u_-;W_n)
\end{align*}and similarly for $ \tilde \b _{[i]}^j(I).$}
The content of the following lemma is to show that both values are equal with high probability.

\begin{lemma}
\label{lm:proba-topo}
Let $ \delta _j = 1 + \frac13 \| i-j \| ,\e >0$. Then,  
\begin{align}
\label{eq:lm-topo-proba}
	\mathbf P( {  \exists n:} \b _{[i]} {(I;W_n)}\ne \tilde \b _{[i],j} { (I;W_n)})\le c_{\e }\min( \delta_j ^{d/2-{\blue \eta} + \e  }\;,\; | I | ).
\end{align}
\end{lemma}

 We prove separately that the LHS is bounded by each of the terms in the minimum in the RHS. First, we have the trivial bound, exploiting the fact that $ \b_{[i]}$ and $ \tilde \b_{[i],j}$ have the same law,
\begin{align*}
\mathbf P( { \exists n:}\b_{[i]}\ne \tilde \b_{[i],j})\le &\mathbf P({ \exists n:}\b_{[i]}\ne 0\text{ \rm or }\tilde \b_{[i],j}\ne 0)\\
\le& 2\mathbf P( { \exists n:} \b_{[i]}\ne 0)\\
\le & 2\mathbf P(Y(Q_i\times I)\ne \es )\le 2\mathbf E(\#Y(Q_i\times I)).
\end{align*}
The bound by $  | I | $ hence follows from Lemma \ref{lm:KR}, invoking also Proposition \ref{prop:ND-strong}.

Bounding by the first term in the minimum in Lemma \ref{lm:proba-topo} is much more tricky.
We see the resampling as a continuous temporal evolution that leads from $ \b _{[i]} $ to $ \tilde \b _{[i]}^j $. We interpolate between the original field $F=F_0:=F(\cdot ;\W )$ and the resampled field $F_1:=F(\cdot ;\W ^{(H_{i,j})})$. The resampling around $j$ is done continuously through the evolution $
F_t=F+t\Delta ,t\in [0,1],$ where we recall that $ \Delta : = \Delta _{i,j}$
is small around $i$ (Proposition  \ref{prop:bd-Delta}).
 The proof of Lemma \ref{lm:proba-topo} resides in the idea that the topology of $ \{F\ge u\}\cap Q_i$ is the same as $ \{F+\Delta \ge u\}\cap Q_i$ if no ``topological event'' occurs during the evolution $ t\to F_t$. It is formalized by the ``deterministic'' Lemma \ref{lm:deter-topo} below.

  {We must first control the size of the connected component of $ Q_i.$
  Let $ m = \delta _j^{\e /3d}$.  Let $ Q_i^m$ the cube with faces parallel to axes centred in $ i$ with sidelength $ m $. {By Assumption \ref{ass:perco}, there is $ \xi >0$ such that
\begin{align*}
\mathbf P( \mathsf{ diam}(\mc C(\{F\ge u-\xi \})  {,Q_i})\ge m)\le c\delta _j^{-\eta + d/2}.
\end{align*}  Introduce the events
\begin{align*}\Omega _1:& = \{\mc C(\{F\ge  { u-\xi }\},Q_i)\su Q_i^m\}\\
\Omega _2 : &= \{\sup_{x\in Q_i^m}\|\Delta (x)\|< {\xi }\}.
\end{align*}
If $ \Omega _1,\Omega _2$ are satisfied, indeed $ F +t \Delta \ge F - \xi $ on $ Q_i^m$, hence for $ t\in [0,1],$
\begin{align*}(\{F + t\Delta >u\}\cap Q_i^m)\su (\{F>u-\xi \}\cap Q_i^m)
\end{align*}
and
\begin{align*}\mc C(\{F + t\Delta \ge u \},Q_i)\su \mc C(\{F\ge u-\xi \},Q_i)\su Q_i^m.
\end{align*}
These events are indeed   {realized with high probability} using also Proposition \ref{prop:concentr}:
\begin{align*} \mathbf P(\Omega _1^c) + \mathbf P(\Omega _2^c)
\le & c \delta _j^{-\eta + d/2} \end{align*}
where the constant depends on the law of $ F$ and $u,\xi ,\e $.
}

  \begin{lemma}
  \label{lm:deter-topo}Recall $ I = [u_-,u_+].$ 
 Assume that $ \Omega_1,\Omega _2$ and the following hold:
 \begin{enumerate}
\item  {There is no stratified quasi critical point $ (t,x)
$  for $ t\in [0,1],x\in $ 
$ Q_i$, or $x\in   Q_{i}^{m} $  {at level }
$u_-$ or $u_+$.}
\item There is no $(t,x)\in [0,1]\times \partial Q_i$ such that $ \nabla F_t(x) = 0$.
\item {\blue For all $ t\in [0,1]$}, there are no two critical points $ x\ne y\in Q_i^m$ of $ F_t$ such that $ x_1= y_1$.
 \item { For $ (t,x)\in [0,1]\times  Q_i^m, \det \Hess_{F_t}(x)\ne 0$ if $ \nabla F_t(x) = 0$. }
\end{enumerate}
Then, $$\b _{ [i]}^{ }(I { ;W_n})= \tilde \b _{[i],j }(I { ;W_n}).$$

 \end{lemma}

 \begin{proof}

  { Denote by $ y_1(t),\dots ,y_{p(t)}(t)$ the critical points of $\mc C(\{F_t\ge u_-\},Q_i)\su Q_i^m$ at time $ t $  { such that $ F_{t}(y_k(t))\in I$}. Then,
	 \begin{enumerate}
		 \item [(i)]
  $ y_k(t)$ satisfies the equation
  $
\nabla F_t(y_k(t)) = 0
$,
\item [(ii)] the Jacobian matrix of $y\mapsto  \nabla F_t$ is
	$
J_t(y) = \left(
\partial _j \partial _iF_t(y)
\right) = \Hess_{F_t}(y),
$
\item [(iii)]
and by (4), $ \det J_t(y_k(t))\ne 0$.
	 \end{enumerate}
	 Hence, by the Inverse Function Theorem, $t\mapsto y_k(t)$ can be $ \mc C^1$ extended on a neighborhood of $ t$, which means the $ y_k(t)$ are $ \mc C^1$ trajectories $ [0,1]\to Q_i^m$.

  The number of critical points could in principle depend on $ t$, but since the component is contained in $ Q_i^m$ during the evolution process they cannot escape $ Q_i^m$, and due to the previous lemma, points stay all along the evolution and $ p(t) = p = const.$   
   {   The corresponding values $ F_t(y_k(s))$ remain in $ I$ thanks to point (1) and by continuity in $ t$.}} {Conversely, there cannot be a critical point $ y\in \mc C(\{F_t\in I\},Q_i)^{ c}$ that enters the component at some time $ t$, because that would mean $ \nabla F_{ t}(y) = 0,F_{ t}(y)\in \{u_{ -},u_{  + }\}$ and hence contradict (i).}

 According to Lemma \ref{lm:Morse}, point (1)    yields an isotopy $ \Gamma $ between connected components
  of
 $ \{F_{t}\ge u\}$
  { in $ Q_{i}^{m}$ at times $ t = 0$ and $ t\in [0,1]$. Those of these components that touch $ Q_{i}$ at time $ 0$ are the same than at time $ t$, because there is no stratified quasi critical point on the boundary  $ \partial Q_{i}.$ Let us label such components at time $ t = 0$ by $ C_{1},\dots ,C_{\ell}$, and   $ C'_1 = \Gamma (C_1),\dots ,C'_{\ell} = \Gamma (C_{\ell})$   the components of $ \{F_t\ge u\}$ touching $ Q_i$ at some time $ t\in [0,1]$. }
 We used that since $ \Omega _1,\Omega _2$ are satisfied, all these components are contained in $ Q_i^m$. In particular, the isotopy yields that $ \b(C_k) = \b(C'_k)$ for $ 1\le k\le \ell$. By point (3), two critical points cannot exchange order in the lexicographic order. Hence, we can write that the reference points are $ y_{i_1}(t),\dots ,y_{i_{\ell}}(t)$ for some fixed indexes $ i_1,\dots ,i_{\ell}\in \{1,\dots ,p\}.$ Hence,
\begin{align*}\b _{ [i]}(u_- { ;W_n})- \tilde \b _{[i]}^{j}(u_-{ ;W_n}) =&\sum_{k = 1}^{\ell }
\b (C_k)\mathbf1_{\{y_{i_k(0)}\in Q_i { ,C_k\su W_n}\}}-\b (C'_k)\mathbf1_{\{y_{i_k(1)} \in Q_i { ,C_k'\su W_n}\}} \\
=& \sum_{k = 1}^{\ell }\b (C_k)(\mathbf1_{\{y_{i_k}(0)\in Q_i, { C_k\su W_n}\}}-\mathbf1_{\{y_{i_k}(1)\in Q_i, { C'_k\su W_n\}}})
\end{align*}
(and a similar representation holds for $ \b _{[i]}(u_+)- \tilde \b _{[i]}^j(u_+)$). Since (2) is satisfied, $ y_{i_k}(t)$ stays at a positive distance from $ \partial Q_i$, hence $ 1_{\{y_{i_k}(0)\in Q_i\}} = 1_{\{y_{i_k}(t)\in Q_i\}}$.

Finally, { since there is no tangency point of some component $ C_k,C'_k$ with the boundary of $ W_n$ by (3), the status $ C_k\su W_n$ or $ C_k'\su W_n$ cannot change},  and the previous sum vanishes.

 \end{proof}

 \newcommand{\B}{  \mathsf b}

To complete the proof of Lemma \ref{lm:proba-topo}, we hence have
\begin{align*}\mathbf P( { \exists n :}\b _{[i]}\ne \tilde \b _{[i]}^{j})\le &\mathbf P(\Omega _1^c) + \mathbf P(\Omega _2^c) +\sum_{ i = 1}^{ 4}\underbrace{\mathbf P((i)\textrm{ is not satisfied })}_{ = :B_k}\\
\le &c\delta _j^{\eta -d/2} +  B_1 + B_2 + B_3 + B_{4}.
\end{align*} Let us estimate the $ B_i$'s.
  For point (1), consider a facet $ \f$ of $ Q_i,Q_i^{ m}$ or $ \partial W_{ n}\cap Q_i^{ m}$.  There is no stratified quasi-critical point on $ \f$ at level $ u$ if  \begin{align*}
 | F(x)-u | >\| \Delta F(x)\|\textrm{ or }\| \nabla _{\f}F(x)\|>\| \nabla _{\f}\Delta (x)\|, x\in \f.
\end{align*}
Similarly, point (2) is implied by
\begin{align*} { \|\nabla F(x) \| >\| \nabla \Delta (x)\|,x\in \partial Q_i }.
\end{align*}
Fix $ \f,u\in \{u_-,u_{  + }\}$. Let
\begin{align*}\B_1  =  & \mathbf P(\exists x\in \f: | F(x)-u | < | \Delta (x) | , \| \nabla _{\f}F(x) \| < \| \nabla_{\f} \Delta (x) \| )\\
\B_2 = & \mathbf P(\exists x\in \partial Q_i: | \nabla F(x) |< \| \nabla \Delta (x) \| )\\
\B_3 = &\mathbf P(\exists t\in [0,1],\exists x\ne y \in Q_i^m \textrm{ such that } x_1= y_1 \textrm{ and } \nabla F_t(x) = \nabla F_t(y) = 0)\\
 {\B_4 =} & { \mathbf P(\exists (t,x)\in [0,1]\times Q_i^m: \nabla F_t(x) = 0, \det \Hess_{F_t}(x) = 0)}
\end{align*}
so that the $ B_i$ are bounded bythe  $ \B_i$. To conclude the proof of Lemma  \ref{lm:proba-topo}, we prove that each $ \B_{ i}$ is in $ O(\delta _j^{ d/2-\eta  + \e })$.
The strategy consists in bounding the probability that  the minimum of a $ \R^{ p + 1}$-valued  Gaussian field on a compact of $ \R^p$ is small. It is a quantitative version of Bulinskaya's lemma \cite{Bulinskaya} in the spirit of the Nazarov and Sodin's version \cite{NazSod}, applicable also to some non-Gaussian fields  { such as $ (t,x)  \mapsto ( \nabla F_{ t}(x), \det \Hess_{F_{t}}(x))$}.

 \begin{lemma}
[Bulinskaya]
\label{lm:bulinsk}  
Let \( p'<p \) be integers, \( A \) a bounded Borel  set of \( \R^{p'} \). For $ 1\le i\le p,$ let \( g_i:A\to \R^{d_i} \) be a smooth centred Gaussian field  {such that\begin{itemize}
\item $ \|g_i(x)\|$ and $ \| \nabla g_i(x)\|$ are bounded  by $ \| (\partial _{ \a }F(x);\a \in I_i)\|$ for some finite $ I_i$,
\item   $G(x) := (g_i(x))\in \R^{p^{*}}
$ has a uniformly bounded density, where \( p^{*} = \sum_id_i.\)
\end{itemize} }
Let \( h_i:\R^{d_i}\to \R \)  smooth functions such that $ h_{ i}$ and its derivative are bounded by $c_{ l} \|x\|^{ l}$ for some $l>0$, and \( f_i(x): = h_i(g_i(x)),1\le i\le p. \) Assume also that each \( h_i(X) \) has a density bounded on \( [-\e, \e ] \) for some \( \e >0 \) on the input \( X\sim \mathcal N(0,I_{d_i}) \). Let
 \(a\in (0,1) \) and
\begin{align*}\psi (x) = (f_i(x))_{i = 1}^p\in \R^p.
\end{align*}  There is \( c_a ,\tau_a >0 \) such that for \( \tau \in (0,\tau_a ) \), 
\begin{align*}
\mathbf P(\inf_{x\in A}\|\psi (x)\|\le \tau )\le c_a  (1 +  \text{\rm{diam}}(A))^{2d} \tau ^{a }.
\end{align*}
\end{lemma}
 {In the applications of this lemma, except for \( \B_4 \), we take $ h_i(s) = s-u$  for some $ u\in \{0,u_-,u_{  + }\}$, hence the assumption on the $ h_{ i}$ is satisfied. In cases $ \B_1 - \B_{ 2},$ we  apply the lemma with $ g_i(x) = (\partial _{ \a_1 }F(x),\dots ,\partial _{ \a_{d_{ i}} }F(x))$ for some multi-indexes $ \a_k ,$ over $ A = \f$ some facet of a cube in $ \R^{ d}$,  hence the density of $ G(x)$ is bounded uniformly  by Proposition  \ref{prop:ND-strong}.

In all cases, detailed below, we first establish $ \B_{ k} \le \mathbf P (  \inf_{  x\in A}\| \psi  (x)\|< c\|\vp \|_A)$ for some other field of the form $ \vp (x) = ( \partial _{ \a }\Delta (x),\a \in I)\in \R^p$ for some set $ I$ of ordered multi-indices, and $ c>0$. Then put $ \sigma   = \delta _j^{ d/2-\eta    }$  and $ a ' = \sqrt{a }<1$. For $ s,t\in \mathbb{R}, s\leqslant t$ implies that $ s\leqslant \sigma ^{ a'}$ or $ t\geqslant \sigma ^{ a'}$. Hence we have for $  \sigma  <\tau_{a'},$ i.e. for $ j$ sufficiently large, using Bulinskaya's lemma and Proposition  \ref{prop:bd-Delta},
\begin{align}
	\notag\B_{ k}\le \mathbf P (\inf_{x\in A}\|\psi (x)\|<c\|\vp \|_A)\le &\mathbf P ( \inf_{ x\in A}\|\psi(x)\|<\sigma  ^{ a' }) + \mathbf P (c\|\vp  \|_A>\sigma ^{ a' })\\
	\notag\le & c_{ a' }(1 + \text{\rm{diam}}(A))^{ 2d}\sigma  ^{ a } + c_{ d}(1 +  { \rm diam}(A)^d)    e^{-  { c}    \sigma ^{ 2a' } \delta _j^{ -d + 2\eta  }/2}\\
\label{eq:bulinsk-field}\le &c_{ a' }(1 + m)^{ 2d}\delta _j^{ a(d/2-\eta )} + c_{ d}(1 + m^d)e^{-  { c}   \delta _j^{ (2\eta -d)(1-a' )   }/2}.
\end{align}
Recalling $ m = \delta _j^{ \e /3d}$, and choosing $ a  $ such that $ a (d/2-\eta ) +(d + 2) \e/3d <(d/2-\eta  + \e )$, the right hand side is in $ O(\delta _j^{ d/2-\eta  + \e })$.
 This will conclude the proof of Lemma \ref{lm:proba-topo}.
Let us give the details.
 }

\begin{enumerate}
\item [$( \B_1)$]  Lemma \ref{lm:bulinsk} can be applied on a facet  \( A = \f \subset  \R^{p'} \) with $p = p'+ 1$, $h_i(s) =s\textrm{ for }i\le p',h_p (s) = s-u$,  \(\psi (x) = (\partial _1F(x),\dots ,\partial _{p'}{F(x)},F(x)-u)\). The Gaussian field \( G(x) = (\nabla _{ \f}F(x),F(x)) \) is non-degenerate for each \( x \) with Proposition \ref{prop:ND-strong}, hence since \( A \) is compact and \( x\to \det \Cov \left(G(x)\right) \) is continuous, \( G(x) \) has a uniformly bounded density.

Let \( \vp (x) = (\nabla_{ \f} \Delta (x),\Delta (x)) \). Bulinskaya's lemma yields with  \eqref{eq:bulinsk-field}
\begin{align*}
\B_2\le& \mathbf P(  \exists  x\in A:\|\psi (x)\|\le 2\|\vp (x)\|,x\in Q_i^m) \le c\delta _j^{ d/2-\eta  + \e }.
\end{align*}

\item[($ \B_2$)] Reason similarly with \( p '=d-1,p= d \), \( \psi (x) = \nabla F(x),h_i(s) = s,\vp (x) = \nabla  \Delta (x) \).
\item[$ (\B_3)$]
 For \( (x,\hat y)= (x,(y_2,\dots ,y_d))\in \R^{2d-1}, \) denote by \( \tilde y= (x_1,y_2,\dots ,y_d) \). Let \( p' = 2d-1 \) and $$ A= \{(x,\hat y)\in Q_i^{m }\times \R^{d-1}: \tilde y\in Q_i^{m },x\ne \tilde y\}.$$
Define the operator \( D^2F(x,\hat y)=\|x- \tilde y\|^{-1}(\nabla F(x)- \nabla F(\tilde y)) \) for \( x\ne \tilde y \).
 We want to estimate the probability that there is some \( t\in [0,1] \) such that
\begin{align*}
\nabla F(x)+t \nabla \Delta (x)= \nabla F(\tilde y)+t \nabla \Delta (\tilde y)= 0.
\end{align*}
It implies (technique of the divided differences)
\begin{align*}
\| \nabla F(x)\|\le \| \nabla \Delta (x)\|,\|D^2F(x,\hat y )\|\le \|D^2\Delta (x,\hat y)\|.
\end{align*}

We wish to apply Bulinskaya's lemma  with \( p' = 2d-1,p = 2d \) to the field
\begin{align*}
\psi (x,\hat y)= (\nabla F (x),D^2F(x,\hat y))\in \R^{2d}.
\end{align*}
		{Before justifying the assumption, we have the Hessian bound $$ \|D^2F(x, \hat y)\|\le   \|\Hess_{ \Delta }\|_A\leqslant \| (\partial ^{ \alpha } \Delta )_{  | \alpha  | \leqslant 2}\|.$$ 
		Hence, its derivatives supremum has finite moments of every order.  Also, by  \eqref{eq:bulinsk-field}
		\begin{align*}\B_3\le \mathbf P(\inf_{(x,\hat y)\in A}\|\psi (x,\hat y)\|\le \|( \nabla \Delta ,  \Hess_{ \Delta })\|_{ Q_{ i}^{ m}})\le c\delta _j^{ d/2-\eta + \e  }.
\end{align*}}

 It remains to check the non-degeneracy condition. Since \( A \) is not closed, the non-degeneracy of each \( \psi (x, \hat y) \) is not enough. Reason by contradiction: assume that for some sequence \( (x_n, \hat y_n)_n \) in \( A \) the max density of \( \psi (x_n,\hat y_n) \) grows to \(\infty \). By compactness, up to taking a sub-sequence, \( (x_n, \tilde y_n)\to (x_{\infty },\tilde y_{\infty })\in  (Q_i^{m })^2.\)
 \begin{itemize}
\item First possibility: \( x_{\infty }\ne \tilde y_{\infty } \), hence the law of \( (\nabla F(x_{\infty }), \nabla F( \tilde y_{\infty })) \) is degenerate, it contradicts Proposition \ref{prop:ND-strong}.
 \item Second possibility: \( x_{\infty }= \tilde y_{\infty } \). Up to taking a sub-subsequence, by compactness of \( \mathbb S ^{d-1} \), there is a unit vector \( z \) such that \( (x_n- \tilde y_n)\|x_n- \tilde y_n\|^{-1}\to z \). Therefore \( D^2F(x_n, \hat y_n)\to \partial _{z} \nabla F(x_{\infty }).\) It follows that \( (\nabla F(x_{\infty }),\partial _{z}\nabla F(x_{\infty })) \) is degenerate, which again contradicts Proposition \ref{prop:ND-strong}.
\end{itemize}

 \newcommand{\triple}[1]{ \vvvert #1\vvvert}
 \item[$ (\B_4)$] 
	 {	 Henceforth, we let \( \| \cdot \| \) be the Frobenius norm on the space of \( d\times d \) matrices. Note that we could replace the norm by a different one at the cost of changing the multiplicative constants in the argument below. 		 Then, }for  matrices $ H,D$, for some $ c_p,c_p'<\infty, $
\begin{align*}
 | \det(H + D)-\det(H) | &\le  c_p\|D\|\sup_{t\in [0,1]}  \|D\|\sup_{ t\in [0,1]}\sum_i | \textrm{ cof }(H + tD)_{i,i} |\\
 \le & c_p'\|D\|(\|H\| + \|D\|)^{d-1},
\end{align*}
 {where \( \textrm{ cof }(H)_{ i,i} \)  is the cofactor of matrix  \( H\) at line $ i$ and column $ i.$}
The latter inequality yields that \( \det \Hess_{F_t}(x) = \det(\Hess_{F}(x) + t\Hess_{\Delta }(x)) \) can only vanish for some \( t\in [0,1] \) if
\begin{align*}
 | \det(\Hess_{F}(x)) | < c \|\Hess_\Delta \|_{Q_i^{m }}(\|\Hess_{F}\|_{Q_i^{m }} + \|\Hess_\Delta \|_{Q_i^{m }})^{d-1}.
\end{align*}
Hence, $ \B_5$ is upper bounded by
\begin{align*}
 \mathbf P\left(\exists x\in Q_i^m: \nabla F(x)\le \| \nabla \Delta (x)\|,\det \Hess_{F}(x) | < c\|\Hess_\Delta \|_{Q_i^{m }}(\|\Hess_{F}\|_{Q_i^{m }} + \|\Hess_\Delta \|_{Q_i^{m }})^{d-1}\right).
\end{align*}
 { We still invoke Bulinskaya's lemma with \( p ' = d,p = d + 1,A =  Q_i^{m } \),
\begin{align*} (g_1(x),\dots ,g_d(x)) = &\nabla F_t(x),\\
g_{d + 1}(x) =& (\partial _{i,j}F(x),i\le j)\in \R^{d(d + 1)/2},\\
h_i(s) =& s,s\in \R,i\le d,\\
h_{d + 1}((a_{i,j})_{i\le j}) =& \det((a_{i,j}) )
\end{align*} hence \begin{align*}\psi (x) =& (\nabla F(x),\det \Hess_{F}(x)).
\end{align*}
By Proposition \ref{prop:ND-strong}, \(G(x): = (g_i(x))_{1\le i\le d + 1} \) has a non-degenerate distribution, and \( A =  Q_i^{m } \) is compact.
The fact that \( h_{d + 1}(X_{i,j}) = \det((X_{i,j})) \) has a bounded density for i.i.d.  Gaussian entries \( X_{i,j},i\le j \) follows from Lemma \ref{lm:det-dens} below.
We hence have for \( a <1 \) some \( c \) such that, by Lemma \ref{lm:bulinsk},
\begin{align*}p_{\tau }: =  \mathbf P( \exists x\in Q_i^{m } : \|\nabla F (x)\|<\tau ,| \det \Hess_{F}(x) | <\tau )\le c\tau ^{a },\tau >0.
\end{align*}
We then have as in  \eqref{eq:bulinsk-field}, with $ \sigma  = \delta _j^{a' (d/2-\eta )}$ for any $ a'\in (0,1),\varepsilon '>0,$
\begin{align*} \B_4\le &c \delta _j^{ d/2-\eta + \e  }   +\mathbf P (\inf_{ x\in Q_{ i}^{ m}}|\det \Hess_{F }(x)|<\delta _j^{ a' (d/2-\eta )}) \\
&+  \mathbf P( \|\Hess_\Delta \|_{ Q_{ i}^{ m}}(\|\Hess_{F}\| + \|\Hess_\Delta \|_{ Q_{ i}^{ m}})^{d-1}>\delta _j^{ a'(d/2-\eta )}) )\\
\le &c'\delta _j^{ d/2-\eta  + \e } +  \mathbf P(\|\Hess_{\Delta }\|_{ Q_{ i}^{ m}}>\delta _j^{ a'(d/2-\eta ) + \e '}) + \mathbf P(\|\Hess_{F}\|_{ Q_{ i}^{ m}}>\delta _j^{ \e' })\\
\le &c\delta _j^{ d/2-\eta + \e  }
\end{align*}using Proposition \ref{prop:concentr} for the last term,
where \( \e ' \)  is such that $ a'(d/2-\eta ) + \e '<d/2-\eta $.

Finally, all terms have been dealt with and give a contribution whose magnitude is not larger than the RHS of \eqref{eq:lm-topo-proba}, thereby concluding the proof.

}
\end{enumerate}

 \begin{proof}[Proof of  lemma  \ref{lm:bulinsk}]Let $ 1\leqslant i\leqslant p.$
	 If for some ${x_0 \in A} , \max_{i \le p}| f_i(x_{0}) | <\tau $, then,
\begin{align*}
	\max_{y\in B(x_0,\tau )\cap A} | f_i(y) | \le \tau (1+\| \nabla f_i\|_{A}).
\end{align*}

 Let $ \g \in (0,1)$, and
\begin{align*}
X= \int_A\frac1{\prod_{i = 1}^p | f_i(x) | ^{\g }}dx.
\end{align*}
Let $ \Omega = \{\inf_{x_0\in A}\|\psi (x_0)\|\le \tau \}.$
 {Up to working instead on $   A + B(0,1)  := \{a + y;a\in A,y\in B(0,1 )\}$, assume without loss of generality  that for all $ y\in A,$ $  | A\cap B(y,\tau  ) | \geqslant \kappa _{ p'}\tau ^{ p'}$ for some $ \kappa _{ p'}>0$ and $ \tau \le  \frac{ 1}2   .$ Let $ B = \prod _{ i} (1 + \| \nabla f_{ i}\|_A)$.}
 Then if $ \Omega $ is realised,
\begin{align*}
X\ge \int_{A\cap B(x_0,\tau )}\frac1{\prod_{i = 1}^p | f_i(x) | ^{\g }}dx\ge   {\kappa _{ p'}} \tau ^{p'}\prod_i\frac1{\tau ^{\g }(1+\| \nabla f_{ i}\|_{A})^{\g }}\geqslant \kappa _{ p'}\tau ^{p'-p\g }\frac1{B^{ \g }}.
\end{align*}
 {By the assumptions on $ g_{ i}$ and $ h_{ i}$, Proposition \ref{prop:concentr} yields  for $r > 1$ that
 $
	\mathbf E (B^{ r\g })\le c(1 + \text{\rm{diam}}(A)^d)$
.}
Let $ r' = (1-r^{-1})^{-1}.$
Hence, with  {Markov's and} H\"older's inequality
\begin{align*}
\mathbf P(\Omega )\le& \tau ^{p\g -p'}\mathbf E(B^{ \gamma }X)\leqslant  \tau ^{p\g -p'}\mathbf E(B ^{ \gamma r})^{\frac1q}\Big[
\mathbf E\Big(
\int_A\prod_i | f_i(x) | ^{-\g }dx
\Big)^{r'}
\Big]^{1/r'}\\
\le &c\tau ^{p\g -p'}(1 + \text{\rm{diam}}(A))^d\Big[
 | A | ^{r'}\mathbf E \Big(
\int_A\prod_i | f_i(x) | ^{-\g r'}\frac{ dx}{ | A | }
\Big)
\Big]^{1/r'}\\
\le &c\tau ^{p\g -p'} (1 + \text{\rm{diam}}(A))^{ 2d} (\sup_{x\in A} \mathbf E(\prod_i | f_i(x) | ^{-\g r'}))^{1/r'}.
\end{align*}
Since $ p\ge p' + 1,$ we can choose $ \g $ such that $ p\g -p'= a $ and then $ r$ such that $ \g r'<1$.

By assumption, the covariance matrix of $ G(x)= (g_i(x))_{i = 1}^m\in \R^{p^{*}}$ is not degenerate for each $x\in A\su \R^{m'}$, it is also a continuous function of $ x$ because the field is smooth, hence its determinant is a non-vanishing continuous function, hence uniformly (in $ x$) bounded from below, and the density of $ G(x)$ is uniformly bounded from above: there is $ C$  {depending on the density} such that for $ X_1,\dots ,X_p$ independent Gaussian vectors with $ X_i\sim \mathcal N(0,I_{p_i})$, for all $ \vp\ge 0 ,$ for all $ x,$
\begin{align*}
\mathbf E(\vp (G(x)))= \int_{\R^{m^{*}}}\vp (y)d\mathbf P_{G(x)}(y)\le C\mathbf E(\vp (X_1,\dots ,X_{p })).
\end{align*}
In particular,
\begin{align*}
\sup_{x\in A}\mathbf E(\prod_{i = 1}^p | f_i(x) | ^{-\g q}) = &\sup_{x\in A}\mathbf E(\prod_{i = 1}^p | h_i(g_i(x)) | ^{-\g r'})\\
\le& C\mathbf E(\prod_{i = 1}^p | h_i(X_i) | ^{-\g r'})\\
 = & C\prod_{i = 1}^p\mathbf E( | h_i(X_i) | ^{-\g r'})\\
\le& C'(1 + \int_{[-1,1 ]} | u | ^{-\g q'}\d u)^p<\infty .
\end{align*}
 \end{proof}

 \begin{lemma}
 \label{lm:det-dens}
 Let a random symmetric matrix $ M$ with centred Gaussian entries in the upper diagonal forming a non-degenerate Gaussian vector. Then $ \det(M)$ has a bounded density in $ 0$.
 \end{lemma}

 \begin{proof}
 Let $ G = (M_{i,j},i\le j)$ be the centred Gaussian vector forming the entries, and let $ G' = (M'_{i,j})$ the GOE model, i.e. the $ M'_{i,j}$ are independent with $ \Var(M'_{ii}) = 2,\Var(M'_{i,j}) = 1,i<j$. The non-degeneracy condition yields that the density $ f_{G}$ is uniformly bounded by $ cf_{G'}(c'\cdot )$ for some constants $ c,c'$, hence it is enough to prove the result for the GOE model.

We prove by induction that there is $C_d<\infty$ such that, if $M_d$ is the $d$-th order GOE, $P(|\det(M_d)|<a)<C_d a$ for $a>0$.
We use the classical result of random matrix theory that the eigenvalues $(l_1,\dots ,l_d)$ of $M_d$ have the explicit joint density $c_d \prod_{i<j} |l_i-l_j|f(l_1)\dots f(l_j)$ for some Gaussian density $f$.
Hence, we can explicitly write

\begin{align*}
	&P(|\det(M_{d+1})|<a)\\
&\quad	=c_d\int_{\R^{n+1}}1_{|l_1\dots l_{d+1}|<a}\prod_{i<j}|l_i-l_j|f(l_1)\dots f(l_{d + 1})dl_1\dots dl_{d + 1}\\
 &\quad = 2c_d\int_0^{\infty }\left(
\int_{ \R^n} 1_{|l_1\dots l_{d+1}|<a}\prod_{i<j} | l_i-l_j | f(l_1)\dots f(l_{d })dl_1\dots dl_d
\right)f(l_{d + 1})dl_{d + 1}
\end{align*}
using the symmetry of $ f$.

Up to a combinatorial term, we can reduce the multiple integral to $(d+1)$-tuples such that $  l_{d+1} > | l_{d + 1} | >\dots > | l_1 | >0$, and we have the crude bound $|l_{d+1}-l_i|<2 l_{d+1} $ for $i<d+1$.  We then have, using the induction hypothesis,

\begin{align*}
P(|\det(M_{d+1})|<a)\le & c_d'\int_0^\infty l_{d+1}^d \int{}1_{ | l_1\dots l_d | <a/l_{d+1}}\prod_{i<j\le d}|l_i-l_j|f(l_1)\dots f(l_{d+1})dl_1\dots dl_{d+1}\\
=&c_d'\int_0^\infty l_{d+1}^d P(|\det(M_d)|<a/l_{d+1})f(l_{d+1}) dl_{d+1}\\
\le & c'_d \int_0^{\infty }l_{d+1}^d C_d \frac{a}{l_{d+1}}f(l_{d+1})dl_{d+1}\\
\le & c''_d a  \int_0^{\infty}l^d\frac{f(l)}{l}dl\\
\le & C_{d+1}a.
\end{align*}
 \end{proof}

  %
%
%
%
\section{Proof of the multi-variate CLT, Theorem \ref{thm:fix_clt}}
\label{sec:fix_clt}
Now, we prove the fixed-level CLT from Theorem \ref{thm:fix_clt}. We proceed in two steps. First, in Section \ref{ss:pf_clt}, we prove the CLT asserted in Theorem \ref{thm:fix_clt} with a possibly degenerate variance. Second, in Section \ref{ss:pf_pos}, we show the positivity of the limiting variance in the univariate case.

%
%
\subsection{Fixed-level and multi-variate CLTs}
\label{ss:pf_clt}
Now, we establish the asymptotic normality of the {topological functionals.

Essentially, the proof idea for asymptotic normality at a fixed level relies on the stabilisation and moment arguments from \cite[Theorem 3.1]{yukCLT}. This technique was designed for dealing with functionals from a Poisson point process in a bounded domain, which is inconvenient in the current setting. Indeed, the white noise $\W$ is defined in all of $\R^d$, and modifications at large distances still have a small but non-vanishing effect on the number of critical points in a given domain. This problem also appears in the investigation of the component counts of excursion sets in \cite{bel}. To provide a suitable analog of \cite[Theorem 3.1]{yukCLT} for the setting where the white noise is defined in the entire Euclidean space, \cite[Theorem 3.2]{bel} is established.

{\blue The key elements of the CLT proof in  \cite[Theorem 1.2]{bel} are a stabilisation and a moment condition \cite[Lemma 3.7]{bel}. Once these conditions are established, the arguments from \cite[Theorem 3.2]{bel} also apply in the present setting with the exception of one additional complication coming from an integer constraint. More precisely, in \cite[Theorem 3.2]{bel} takes values in the positive integers and this makes it possible to exploit that $L^2 \ge L$ holds for all $L =0, 1, \dots$. To address this  issue, in Proposition \ref{prop:mom} below, we establish item 3 not only for the exponent $2 + \e$ but also for the exponent equal to 1.
}
\def\uu{\boldsymbol u}

To state the {\blue moment and stabilisation} conditions precisely, we first need to introduce the costs associated with resampling. 
Set $\uu := (u_1, \dots, u_K)$. By Cramer-Wold, it suffices to show asymptotic normality of 
$${\blue \beta _n^{}(\uu; F) := \sum_{i \le K} a_i \beta _n (u_i;F),}$$
where $a_1, \dots, a_K \in \R$ and  where we write  {$ \beta _n (u_i;F)$ as in  \eqref{eq:b_n}}. We set the global variation on $ Q_j$ resampling
\begin{align}
	\label{eq:bdej}
	B_{\De, j}^{(n)} :=\b_n\big(\uu; F\big)  - \b_n\big(\uu; F^{(Q_j)}\big) =  \sum_{i \in W_n}B_{\De, i, j}(\uu, W_n)
\end{align}
 {\blue where $B_{\De,i,j}(\uu, W_n) = \beta _{[i]}(\uu, W_n)- \tilde \beta _{[i],j}(\uu, W_n)$ denotes the contribution to $B_{\De, j}^{(n)}$ coming from critical points in the cube $Q_i$.} To ease notation, we henceforth often write $B_{\De, i, j}$ instead of the more verbose $B_{\De, i, j}(u, W_n)$ when the value of $n$ is clear from the context. 

 $ B_{\Delta ,j}^{(n)}$ represents the total variation when $ Q_j$ is resampled. Recall that $\g := \eta/d -1/2$.
%
%

\bepr[Stabilisation condition; Lemma 3.7 of \cite{bel}]
\label{prop:stab}
Assume that $\g > 1$. Then, the sequence $\{B_{\De, o}^{(n)}\}_n$ converges a.s.~to some almost surely finite random variable $B_{\De, o}^{(\ff)}$.
\enpr
\bep
We want to show that $B_{\De, i, o}(\uu, W_n)=0$ for all $i \in \Z^d$ with $|i| \ge R$ and $n \ge 1$, where $R$ is an almost surely finite random variable. This will be achieved via the Borel-Cantelli lemma. More precisely,  first, Lemma \ref{lm:proba-topo}   implies that
 \begin{align}
	\label{eq:bde}
	 \mathbf P(\cup_{n \ge 1} \{ B_{\De, i, o}(\uu, W_n) \ne 0\}) \in O\big( |i|^{-\g d + \e}\big).
\end{align}
In particular, since $\g > 1$, the products $|i|^{d - 1} \mathbf P(\cup_{n \ge 1} \{B_{\De, i, o}(\uu, W_n) \ne 0\})$ are summable for $i \in \Z^d$. Hence, applying the Borel-Cantelli lemma shows that the existence of the asserted random $R < \infty$ such that $B_{\De, i, o}(\uu, W_n)=0$ for all $i \in \Z^d$ with $|i| \ge R$ and $n \ge 1$.
\enp
{Let  $W_n^+$ be the box with the same centre as $W_n$ but side length doubled. Let $|W_n|$ be the Lebesgue measure of $W_n$.}
Now, we write $d_{j, n} := \dist(j, W_n^+) { +1}$ for the distance of a site $j$ to the window $W_n^+$. 
To ease notation, we assume henceforth that $\mathbf E\big(Y(Q_{\blue 0} \ti \R)^{l_0}\big) < \ff$  for some $l_0 > 2$  { which results from $ F\in \mathcal{C}^{l_0}$ by Theorem  \ref{thm:gass}}.     in the rest of this section. {For the proof of Proposition \ref{prop:mom} below, we assume that $l_0 > 4$ and that $\g > 12$}. Here, we also note that these assumptions imply that $ \gamma >\frac{ 9l_0}{l_0+6}$. 

%
%
\bepr[Moment conditions; {\blue Lemma 3.7} of  \cite{bel}]
\label{prop:mom}
Assume that $l_0 > 4$ and $\g > 12$. Then, for any sufficiently small $\e > 0$, it holds that
\been
\im $\sup_{n \ge 1}\sup_{j \in \Z^d}  \mathbf E(|B_{\De, j}^{(n)}|^{{\blue 2 + \e}}) < \ff$;
\im  $$\sup_{n, k \ge 1}\fr{\sum_{j  \in \Z^d \co d_{j, n} >k}\mathbf E(|B_{\De, j}^{(n)}|^{{\blue 2 + \e}})}{|W_n|^3 k^{-\eta/3}(k^d + k|W_n|^{(d - 1)/d})} < \ff;$$
\im  $\sup_{n \ge 1}|W_n|^{-1}\sum_{j \in \Z^d}\mathbf E(|B_{\De, j}^{(n)}|^{{\blue b}}) < \ff$ {\blue for both $b = {\blue 2 + \e}$ and $b = 1$. }
\enen
        \enpr
	The key step in the proof of Proposition \ref{prop:mom} are the following moment bounds on $B_{\De, j}^{(n)}$. { Henceforth, we set $|(x_1, \dots, x_d)| := \max_{i \le d}|x_i|$ for the $\ell_\ff$ norm in $\R^d$.}

%
        %
\def\r{\rho}

\bel[Moment bound on $B_{\De, j}^{(n)}$]
\label{lem:mom1}
{\blue Let $\e > 0$, $l_0 > 2$, $\g > 1$ and $m \in \{1\}\cup [2 ,l_0\g/ ( l_0 + \g)]$.} Then, 
$\mathbf E(|B_{\De, j}^{(n)}|^m) \in O\big( d_{j,n}^{-\g d(l_0 - m)/l_0 + \e}(|W_n| \wedge d_{j, n}^{d})^m\big).$
\enl

Henceforth, we let
$$g_{j, n}(k) := \big|\big\{i \in \bar W_n \co |i - j| = k\big\}\big| \in O\big( k^{d - 1} \wedge n^{1- 1/d}\big).$$
denote the number of elements of $\bar W_n:=\{i \in \Z^d \co Q_i \ne \es\}$ at distance $k \ge 0$ from $j \in \Z^d$.
\bep
First{\blue,  for $2 \le m < l_0\g/(l_0 + \g)$,} by the Jensen inequality,  
	$
	\mathbf E(|B_{\De, j}^{(n)}|^m)\le \Big(\sum_{i \in \bar W_n} \mathbf E(|B_{\De, i, j}|^m)^{1/m}\Big)^m.
	$
Now, invoking \eqref{eq:bde} and the H\"older inequality with $q' = l_0/m$ and $p' = l_0/(l_0 - m)$ gives that
\begin{align}
	\label{eq:bd4}
	\mathbf E(|B_{\De, i, j}|^m) \le \big(\mathbf E(|B_{\De, i, j}|^{l_0})\big)^{1/q'} \mathbf P(B_{\De, i, j}\ne 0)^{1/p'} \in O\big(|i - j|^{-\g d /p' + \e}\big)
\end{align}
Note that by our moment bound, we have  $\mathbf E(|B_{\De, i, j}|^{l_0}) < \ff$. Now,
$$\sum_{i \in \bar W_n}|i - j|^{-\g d/(p'm)} \le C\sum_{k \ge d_{j, n}}g_{j,n}(k) k^{-\g d/(p'm) + \e},$$
where the right-hand side is in $O\big((|W_n|\w d_{j,n}^d)d_{j, n}^{-\g d/(p'm)+ \e}\big)$, as asserted.

{\blue For $m = 1$, invoking \eqref{eq:bde} and the H\"older inequality with $q' := l_0$ and $p' := l_0/(l_0 - 1)$,
\begin{align}
	\label{eq:bd4}
\mathbf E(|B_{\De, i, j}|) \le \big(\mathbf E(|B_{\De, i, j}|^{l_0})\big)^{1/l_0} \mathbf P(B_{\De, i, j}\ne 0)^{1/p'} \in O\big(|i - j|^{-\g d /p' + \e}\big)
\end{align}
Now,
$\sum_{i \in \bar W_n}|i - j|^{-\g d/p'} \le C\sum_{k \ge d_{j, n}}g_{j,n}(k) k^{-\g d/p' + \e},$
where the right-hand side is in $O\big((|W_n|\w d_{j,n}^d)d_{j, n}^{-\g d/p'+ \e}\big)$, as asserted.

}
\enp

%
%
Finally, we complete the proof of Proposition \ref{prop:mom}.
\bep[Proof of Proposition \ref{prop:mom}]
{ Note that our assumptions on $l_0$ and $\g$ imply that $ 3<l_0\g/(l_0 + \gamma )$.} In particular, part (1) follows from Lemma \ref{lem:mom1}, and we concentrate on parts (2) and (3). In both parts, we write
$g_k^{(n)} :=|\{i \in \Z^d\co \dist(i, { W_n^+}) = k\}|$
for the number of sites that are at distance $k\ge 1$ from $W_n$.
Moreover, putting $\r(\e) := 1 - {\blue (2 + \e)}/l_0$, we also note that for $k \ge n^{1/d}$, Lemma \ref{lem:mom1} and Jensen give that
\begin{align}
	\label{eq:jdj}
	\sum_{j\co d_{j, n}\ge k}\mathbf E(|B_{\De, j}^{(n)}|^{{\blue 2 + \e}}) \le C\sum_{\ell \ge k}g_\ell^{(n)}|W_n|^{{\blue 2 + \e} }\ell^{-\g d\r(\e)}.
\end{align}
Since $g_\ell^{(n)} \in O\big(\ell^{d - 1} \big)$, the right-hand side is of order $O\big(|W_n|^{{\blue 2 + \e}} k^{d-\g d\r(\e)}\big)$.
\medskip

\noindent {\bf Part (3).}
Since the number of $j \in \Z^d$ with $d_{j, n} \le n^{1/d}$ is of order $O(|W_n|)$, it suffices to deal with the case, where $d_{j, n} \ge n^{1/d}$.
Then, by \eqref{eq:jdj} with $k = \lfloor n^{1/d}\rfloor$,
$$|W_n|^{-1}\sum_{j\co d_{j, n}\ge n^{1/d}}\mathbf E(|B_{\De, j}^{(n)}|^{{\blue 2 + \e}}) \in  O\big(|W_n|^{{\blue 2 + \e}   - \g \r(\e) )}\big).$$
Hence, invoking the assumption $\g >3l_0/(l_0 - 2)$ concludes the proof of part (3). {\blue Next, assume $b =1$. Then, as in \eqref{eq:jdj}, 
\begin{align*}
	\sum_{j\co d_{j, n}\ge n^{1/d}}\mathbf E(|B_{\De, j}^{(n)}|) \le C\sum_{\ell \ge n^{1/d}}g_\ell^{(n)}|W_n|\ell^{-\g (1 - 1/l_0)} \in O\big(|W_n|^{2-\g (1 - 1/l_0)}\big).
\end{align*}
Therefore, 
$|W_n|^{-1}\sum_{j\co d_{j, n}\ge n^{1/d}}\mathbf E(|B_{\De, j}^{(n)}|) \in  O\big(|W_n|^{1   - \g (1 - 1/l_0) }\big)$.
Hence, invoking the assumption $\g >3l_0/(l_0 - 2)$ concludes the proof of part (3). 
}
\bigskip

\noindent {\bf Part (2).}
First, \eqref{eq:jdj} implies that for $k \ge n^{1/d}$,
$$\sum_{j\co d_{j, n}\ge k}\mathbf E(|B_{\De, j}^{(n)}|^{{\blue l_0}}) \in  O\big(|W_n|^{{\blue 2 + \e} } k^{d- \g d \r(\e)) }\big).$$
In particular, noting that $\eta = \g d + d/2$, $l_0 > 4$ and  $\g > 6$,   for $ k\ge n^{1/d}$, the ratio
$$\fr{|W_n|^{{\blue 2 + \e} } k^{d -\g\r(\e)d }}{|W_n|^3 k^{d -\eta/3}}= |W_n|^{{\blue 2 + \e} - 3} k^{(1/6 - \g(\r(\e) - 1/3))d}$$
remains bounded.

Second, we consider the case $d_{j, n} \le n^{1/d}$. Then, by Lemma \ref{lem:mom1},
$$\sum_{j\co k \le d_{j, n}\le n}\mathbf E(|B_{\De, j}^{(n)}|^{{\blue 2 + \e}}) \le C \sum_{k \le \ell \le n}g_\ell^{(n)}|W_n|^{{\blue 2 + \e}}\ell^{(1-\g \r(\e))d}.$$
Since $\max_{\ell \le n}g_\ell^{(n)} \in O\big(n^{1 - 1/d}  \big)$, this is of order $O\big(n^{3d - 1 +d\e} k^{1 + (1-\g \r(\e))d }\big)$  .
Again, since $\g > 12$,  the ratio
$$\fr{|W_n|^{3 +\e  - 1/d} k^{1+(1-\g \r(\e))d }}{|W_n|^{4 - 1/d} k^{d -\eta/3}}= |W_n|^{-1 + \e} k^{1 +d/6 -   \g(\r(\e)-1/3)d}$$
remains bounded, thereby concluding the part of part (2).
\enp

%
%
\subsection{Positivity of variance}
\label{ss:pf_pos}
~

In this section, we establish the positivity of the limiting variance $\s(u)^2$ from Theorem \ref{thm:fix_clt}.
The strategy is to proceed similarly as in the case of the component count considered in \cite[Theorem 1.3]{bel}. The idea is that in the proof the martingale CLT from \cite[Theorem 1.3]{bel}, the authors also derive a non-trivial and highly useful representation of the limiting variance. First, we recall from Proposition \ref{prop:stab} that $B_{\De, o}^{(\ff)}$ is the difference of the Betti numbers before and after the resampling of the white noise in $Q_{\blue 0}$. Now, we have
\begin{align}
\label{eq:s2}
	\s^2: = \s(u)^2: = \mathbf E\big(\mathbf E(B_{\De, o}^{(\ff)}\ba \FF_0)^{2}\big),
\end{align}
where
$\FF_0 := \s\big(\W \cap Q_i\co Q_i {\blue \preceq} Q_{\blue 0} \big)$
is the $\s$-algebra generated by the white noise $\W$ in all cubes of the form $Q_i$ {\blue with reference point $ i$  preceding or being equal to the origin  $ 0$}   in the lexicographic order.  This representation provides a starting point for the positivity proof.

We now proceed along the lines of \cite[Theorem 1.3]{bel} to show that $\s^2 > 0$. In order to avoid redundancy, we sometimes sketch the general argument and concentrate on the steps that are markedly different. We note that some of our steps are in fact simpler because {\blue of the percolation Assumption \ref{ass:perco}. We now briefly comment on the assumption that $\mu(u) \ne 0$.}

\begin{remark}
	\label{rem:bet}
	A key step in the positivity proof of \cite[Theorem 1.3]{bel} is the positivity of $\mu(u)$. While for the considered case of connected components this can be enforced by assuming that the support of the spectral measure has a support containing an open set, the situation is more complicated and we plan to proceed as in \cite[Theorem 1.2c]{wigman1} and \cite[Proposition 5.2]{wigman2}.

	Nevertheless, in the case of actual Betti numbers, the positivity of $\mu(u)$ can be verified along the lines of the \cite{wigman2,wigman1}. More precisely \cite{wigman2,wigman1} deal with Betti numbers of level sets whereas we need excursion sets and general topological functionals. Although this change does not cause major differences, we briefly recall the argument from \cite[Lemma 5.5]{wigman2} to make our presentation self-contained. We first consider any smooth function $h \co \R^d \to \R$ with support in some compact set $D \su \R^d$ such that the excursion set at level $u$ has a positive Betti number.  Now, as in Lemma \cite[Lemma 5.5]{wigman2}, since the field $F$ is smooth it allows for a series representation in terms of eigenfunctions, and we can conclude that with positive probability, $h$ is at most $\e$ away from $F$ in $\mathcal{C}^1(D)$-distance. Hence, we conclude from Morse theory in the form of Lemma \ref{lm:Morse} that with positive probability the excursion sets of $F$ and of $h$ are isotopic. Hence, the excursion set $A(u; F)$  has a positive Betti number with positive probability.
\end{remark}

Recall that $\bt(F; u) $ is the Betti number of the  {union of components of the} excursion set $\mathbf E(u)$    in the interior of $W_n$.
Similarly to \cite{bel}, a key step is that the expected functional of the excursion set becomes smaller after a suitable perturbation of the random field.
%
%
\bel[Reduction of expected functional by perturbation]
\label{lem:pert}
Assume that {$l_0 > 2 \cdot 64^2$} and that $\g > 3$. Then,  there exists $m \ge 1$ and a  {nonempty} open set $S\su \R$ such that  {if
$ \mu (u)>0$, then }
$$\sup_{s \in S} \lim_{n\tff} \big(\mathbf E(\bt(F + s(q\star \one_{mQ_0}), u)) - \mathbf E(\bt(F, u))\big) < 0.$$
\enl
 {and if $ \mu (u)<0$, then}
$$\sup_{s \in S} \lim_{n\tff} \big(\mathbf E(\bt(F + s(q\star \one_{mQ_0}), u)) - \mathbf E(\bt(F, u))\big) > 0.$$
We first explain how Lemma \ref{lem:pert} implies $\s^2 > 0$  {if $ \mu (u)>0$, the case $ \mu (u)<0$ is symmetric}. After that, we prove Lemma \ref{lem:pert}.
Since the proof is parallel to that of \cite[Lemma 3.12]{bel}, we only give the main idea.  The positivity proof relies on the representation \eqref{eq:s2}.  In fact, it will be convenient to generalise the definition of $B_{\De, o}^{(\ff)}$ so as to compare the excursion functional after resampling to those obtained by a deterministic perturbation of the considered random field. More precisely, for $w \in \CC^4$ satisfying $\lim_{|x|\tff} w(x) =0$ {\blue and $\lim_{|x|\tff} \nabla w(x) =0$}, we set
$$D(w) :=  \lim_{n \tff} \big(\bt(F + w; u) - \bt\big(F^{(Q_{\blue 0})}; u\big)\big),$$
{\blue where we now argue why the limit exists. 
Note that we cannot directly use Proposition \ref{prop:stab} because this only deals with $w = 0$. However, we can argue similarly and copy the proof.
First, we write 
$$ \big(\bt(F + w; u) - \bt\big(F^{(Q_{\blue 0})}; u\big)\big) = \sum_{i \in \bar W_n}B_{\De, i}(w),$$
 where $B_{\De,i}(w)$ denotes the contribution to the left-hand side coming from critical points in the cube $Q_i$.
We want to show that $B_{\De, i}(w)=0$ for all  $|i| \ge R$ and $n \ge 1$, where $R$ is an almost surely finite random variable. This will be achieved via the Borel-Cantelli lemma as in Proposition \ref{prop:stab}. More precisely,  arguing as in Lemma \ref{lm:proba-topo}   implies that
$$	 \mathbf P(\cup_{n \ge 1} \{ B_{\De, i}(w) \ne 0\}) \in O\big( |i|^{-\g d + \e}\big).$$
Here, the topological lemma is applied to $\De + w$ instead of $\De$. Note that since both $w$ and $\De w$ vanish at infinity, the arguments carry over.
In particular, since $\g > 1$, the {\blue probabilities} $\mathbf P(\cup_{n \ge 1} \{B_{\De, i, o}(\uu, W_n) \ne 0\})$ are summable for $i \in \Z^d$. Hence, applying the Borel-Cantelli lemma shows that the existence of the asserted random $R < \infty$ such that $B_{\De, i, o}(\uu, W_n)=0$ for all $|i| \ge R$.
}

%
%
\bep[Proof of positivity of the limiting variance]
In the proof, we rely on a variant $B_{\De, 0; m}^{(\ff)}$ of $B_{\De, o}^{(\ff)}$, where instead of a partition into side length 1 boxes, we use boxes of side length $m\ge 1$. {\blue Performing this replacement on the right-hand side of \eqref{eq:s2} will give an expression $\s(u;m)^2$ on the left-hand side. This will satisfy a scaling relation with respect to the variance $\s^2(u)$ from \eqref{eq:s2}. More precisely, instead of covering $W_n$ with $O(|W_n|)$ cubes of side length 1, we now cover $W_n$ with $O(|W_n|/m^d)$ cubes of side length $m$. Therefore, $\s^2(u;m) = m^d\s^2(u)$.
}
To avoid confusion, we stress that there is no clash in notation in the sense that $B_{\De, 0; m}^{(\ff)}$ is different from $B_{\De, i, j}$.

Let $Z_0  { :=  \mathcal W(mQ_0)}$, which is normal random variable with variance $m^d$. Then, the $\mathbf E(B_{\De, 0;m}^{(\ff)} \ba Z_0) = G(Z_0)$  for some measurable function $G: \R \to \R$.
Now, the white noise $\W$ on $mQ_0$ decomposes into $Z_0\one_{mQ_0}(\cdot)$ and an orthogonal part $\mathcal W_1$. Hence, we can represent the random field $F$
$$F  = q \star (\mathcal W|_{(mQ_0)^c} + Z_0\one_{mQ_0}(\cdot) + \mathcal W_1),$$
and also, for any fixed $s \in \R$,
$$ q \star (\mathcal W|_{(mQ_0)^c} + (Z_0 + s)\one_{mQ_0}(\cdot) + \mathcal W_1) = F + w,$$
where $w =  s(q\star \one_{mQ_0})$.
In particular, for any fixed $s \in \R$, we have $\mathbf E(G(Z_0 + s)) = \mathbf E(D(s(q\star \one_{mQ_0})))$, so that $\mathbf E(G(Z_0)) = 0$. Moreover, Lemma \ref{lem:pert} implies that $\mathbf E(G(Z_0 + s)) < 0$ for $s$ contained in an open set. Then, the formula for the conditional variance  gives the asserted $ \s{\blue(u; m)}^2 \ge \Var(G(Z_0)) > 0$.
\enp

\def\d{{\rm d}}
It remains to prove the perturbation property in Lemma \ref{lem:pert}. Before starting the proof, recall that we assume that for the critical points, we have moments up to order ${\blue 2 + \e} = 2 + \e$.
\bep[Proof of Lemma \ref{lem:pert}]
	In the proof, we may assume that $\int_{\R^d}q(x)\d x = 1$. Otherwise, one repeats the proof below after $s$ by $s/\int_{\R^d}q(x) \d x$, which is well-defined since $\int_{\R^d}q(x) \d x \ne 0$ by assumption.  Moreover, we also set
	$$
	\b_n\big(A, F): =\sum_{\substack{(x,v)\in \Ys(W_n\times [u,\infty ))\\ x \in A}}\delta ^{ref}(x,v,C_{x}),$$
and the same for $\b_n\big(A, F + w)$.

%
%
The key observation is to realise that $\lim_{u\tff}\mu(u) =0$. Indeed, we first note that $\mu(u)$ is bounded above by the expected number of   critical points above level $u$ per unit volume. Since the expected number of such critical points is finite so that the dominated convergence theorem yields that $\lim_{u\tff}\mu(u) = 0$. {\blue We now give the details of the arguments.}

{
\blue	To show that $\mu(u) \to 0$, we distinguish between the critical points in the interior and at the boundary. First, for the interior points,  almost surely as $u\tff$, 
$$\Ys\big(\mathsf{int}(Q_{\blue 0})\ti[u,\ff)\big) \to  0.$$ 
Moreover, by Theorem \ref{thm:gass}, we conclude that $\mathbf E \big[\Ys\big(\mathsf{int}(Q_{\blue 0})\ti[u,\ff)\big)\big] <\ff$. Hence, by dominated convergence, also the expected value tends to 0. Second, we deal with the points intersecting the boundary. 
\begin{align*}
\mathbf E\big(|\Ys(\partial W_n \ti [u, \ff)) |\big) \le \sum_{i: Q_i \cap   W_n\neq \emptyset } \mathbf E\big(|\Ys(Q_i)|\big).
\end{align*}
Now, the number of summands is of order $o(|W_n|)$, thereby showing that the boundary contributions to $\mu(u)$ can be neglected.
}

{ Let $B(k)$ denote the ball of radius $k$ in the $|\cdot|_\ff$-norm centred at the origin}. Now, we use that $\mu(u) > 0$ for every $u \in \R$. To justify this note that Theorem 1.2c in (Wigman, 2021) shows the positivity $\mu(u)$ in the setting of level sets, which also extends to excursion sets. Hence, we conclude that  there exists $\zeta > 0$ and an open $S \su \R$ such that  $\sup_{s \in S}\mu(u - s) - \mu(u) < -7\zeta$. 
Invoking the definition of $\mu$ therefore gives  for  $ k>k_0$ sufficiently large
\begin{align}
\label{eq:p1}
        \mathbf E\big(\b_n\big(B(k), F + s\big)\big) - \mathbf E\big(\b_n\big(B(k), F\big)\big) &< -6\zeta k^d,
	\end{align}
	where we write $B(k) := B(0, k)$.
Then, as in \cite[Theorem 1.3]{bel}, we proceed in the following four steps which are valid for $n > m^d$, provided that $m$ is chosen sufficiently large. While it would be possible to extract the specific lower bound for $m$ from the proof, it is not needed for our aims.  For such $m$, we set $k:= m - \sqrt m$, $s\in S$    and where  we set $w:= w_{s, m} := sq\star \one_{mQ_0}$ :
\begin{align}
        \mathbf E\big(\big|\b_n\big(W_n, F\big) - \b_n\big(B(k), F\big) - \b_n\big(W_n \sm B(k), F)\big|\big)&\le \zeta m^d \label{eq:p2}\\
        \mathbf E\big(\big|\b_n\big(W_n, F + w\big) - \b_n\big(B(k), F + w\big) - \b_n\big(W_n \sm B(k), F + w)\big|\big)&\le \zeta m^d \label{eq:p22}\\
        \mathbf E\big(\big|\b_n\big(W_n \sm B(k),  F+ w\big) - \b_n\big(W_n \sm B(k),  F\big)\big|\big)&\le \zeta m^d \label{eq:p4}\\
      \mathbf E\big(\big|\b_n\big(B(k),  F+ w\big) - \b_n\big(B(k),  F + s\big)\big|\big)&\le \zeta m^d \label{eq:p3}.
\end{align}
Hence, as soon as \eqref{eq:p2}--{\blue\eqref{eq:p3}} are satisfied, we obtain the desired
$$\sup_{s \in S} \lim_{n\tff} \big(\mathbf E(\bt(F + s(q\star \one_{mQ_0}))) - \mathbf E(\bt(F))\big) \le -\zeta m^d,$$
We now verify the individual claims separately.
\bigskip

%
%
{\bf Bounds \eqref{eq:p2} \& \eqref{eq:p22}.}
We only deal with \eqref{eq:p2} since the arguments for \eqref{eq:p22} are analogous.
The key observation is that the Betti number is additive in the connected components. Hence, the Betti numbers of components contained in $B(k)$ are taken into account in $\b_n(W_n, F)$. Similarly also the Betti numbers of components contained in $W_n \sm B(k)$ are accounted for in $\b_n(W_n, F)$. Hence, the deviations in \eqref{eq:p2} come from components intersecting $\pa B(k)$.

Hence, it therefore suffices to show that $\sup_{i \in \Z^d}\mathbf E\big(\b_n\big(\CC(Q_i,\mathbf E(u))\big)\big) < \ff$.  To prove this claim, we let $\Om_{i, m}$ denote the event that $m \ge 1$ is the smallest integer such that all connected components hitting $Q_i$ are contained in $Q_i^m$. 
Henceforth, we write $Y = Y^{\partial }$   for the set critical and stratified critical points of the field $F$.
Then, by the Cauchy-Schwarz inequality,
\begin{align*}
	\sup_{i \in W_n}\mathsf E\big(\b_n\big(\CC(Q_i,\mathbf E(u))\big)\big)  &\le \sum_{m \ge 1}\mathbf E\big(Y(B(m))\one\{\Om_{i, m}\}\big) \\
	&\le \sup_{i \in W_n}\sqrt{\mathbf E\big(Y(Q_i)^2\big)}\sum_{m \ge 1}|B(m)| \sqrt{\mathbf P(\Om_{i, m})} < \ff,
\end{align*}
where the last step follows since the excursion set is in the subcritical regime.
\medskip

It remains to deal with  \eqref{eq:p4} and \eqref{eq:p3}, where the idea is to apply Lemmas  \ref{lm:proba-topo} and \ref{lm:deter-topo}.  First, we note that  that $\b_n(B(k), F + w)$ decomposes into the contributions coming from the cubes $Q_j$. That is,
$$\b_n(B(k),  F + w) =: \sum_{j \in B(k) \cap \Z^d} B_j(F+ w),$$
and similarly for $\b_n(B(k), F + s)$.  The key step is now to obtain bounds on $\mathbf E(|B_j(F ) - B_j(F + w)|)$, i.e., on the contribution to the perturbation from each cube $Q_j$.

To this end, we first need to show that the field perturbation $w$ is small. Indeed, we have that
\begin{align}
	\label{eq:wb}
	\sup_{x \in \R^d} |w(x)| \le |s| \sup_{x \in \R^d} \int_{B(m)}|q(x - y)|\d y \le |s| \sup_{x \in \R^d} \int_{\R^d}|q(x - y)|\d y< \ff,
\end{align}
and a similar computation shows the boundedness of derivatives up to order 3.

Next, we recall that by Theorem  \ref{thm:gass} on the finiteness of moments of the number of critical points, we have  $\mathbf E(|B_j(F)|^{{\blue l_0}})< \ff$. Moreover, in Remark {\blue 1 of \cite{GassSte}}, it is argued that this moment bound is uniform over all fields with an upper bound on the absolute value of their derivatives. Hence, we deduce that we also have $\mathbf E(|B_j(F + w)|^{{\blue l_0}})< \ff$    uniformly over all considered perturbations $w$.

We will use this observation as follows. The H\"older inequality with $q' =  {\blue 2 + \e}$ and $\qm ={\blue l_0/(l_0} - 1)$ shows that for every $j \in \Z^d$ the expression 	$\mathbf E(|B_j(F ) - B_j(F + w)|)$ is at most
\begin{align*}
	\mathbf P\big(B_j(F) \ne  B_j(F + w)\big)^{1/\qm}\Big({\mathbf E(|B_j(F )|^{{\blue l_0}})^{1/l_0}} +  {\mathbf E(|B_j(F + w)|^{{\blue l_0}})^{1/l_0}}\Big),
\end{align*}
where as argued above, the second factor is of constant order. Therefore,
\begin{align}
	\label{eq:bjfw}
	\mathbf E(|B_j(F ) - B_j(F + w)|)\le  C\mathbf P\big(B_j(F) \ne  B_j(F + w)\big)^{1/\qm}.
\end{align}
Relying on this observation, we now conclude the proofs of \eqref{eq:p4} and \eqref{eq:p3}.
\bigskip

%
%
{\bf Bound \eqref{eq:p4}.}
First,
\begin{align*}
        &\mathbf E(\big|\b_n(W_n \sm B( k),  F+ w) - \b_n(W_n \sm B( k),  F)\big|)\\
        &\le \sum_{j\co k < |j| \le m}\mathbf E(|B_j(F + w) - B_j(F)|)
         +\sum_{j \co |j| \ge m + \sqrt m }\mathbf E(|B_j(F + w) - B_j(F)|)
\end{align*}
Note that the number of summands in the first sum is of order $O(m^{d - 1/2})$.  Hence, according to \eqref{eq:bjfw}, it suffices to show that
\begin{align}
	\label{eq:bjfw1}
	\max_{j\co k < |j| \le m}{\mathbf P(B_j(F + w) \ne B_j(F))}   \in o(1),
\end{align}
and
\begin{align}
	\label{eq:bjfw2}
	 \sum_{j \co |j| \ge m + \sqrt m }\mathbf P(B_j(F + w) \ne B_j(F))^{1/\qm}   \in o(m^d).
\end{align}
In both cases, we rely on Lemmas  \ref{lm:proba-topo} and \ref{lm:deter-topo}. We apply these results to the family of perturbations $F^{(t)} := F + tw$. Then,  for $ x \in B(m)^c$,
$$|w(x)| = \Big|\int_{B(m)} q(x - u)\d u\Big| \le\int_{|u| > \dist(x, B(m))} |g(u)| \d u \in O(\dist(x, B(m))^{d - \b}).$$
In particular, Lemmas \ref{lm:proba-topo} and \ref{lm:deter-topo} give that   $\mathbf P(B_j(F + w) \ne B_j(F)) \in O(m^{d - \b})$, thereby implying \eqref{eq:bjfw1}. Moreover, since $\eta/d > \qm + 1$, we can bound \eqref{eq:bjfw2} by
$$c_2\sum_{j \co |j| \ge m + \sqrt m }\hspace{-.2cm} \dist(j, B( m))^{-(\eta - d)/\qm}\le c_3\sum_{ i \ge m + \sqrt m}\hspace{-.2cm} i^{d - 1} (i - m)^{-(\eta - d)/\qm} \in O(m^{d - 1/2}),$$
thereby concluding the proof of \eqref{eq:bjfw2}.
 \bigskip

%
%
{\bf Bound \eqref{eq:p3}.}
First, arguing as in \eqref{eq:p4}, it suffices to show that
$$\max_{j\co |j|\le k}\mathbf P\big(B_j(F + s) \ne  B_j(F + w)\big) \in o(1).$$
In both cases, we rely again on Lemmas  \ref{lm:proba-topo} and \ref{lm:deter-topo}. We apply this result to the family of perturbations $F^{(t)} := F + s + t(w - s)$. The arguments are now very similar to the proof of \eqref{eq:p3} but to make the presentation self-contained, we give some details. Indeed, we have that  
$$\sup_{x\in B(k)} |s - w(x)| \le |s| \int_{\R^d \sm B(\sqrt m)}|q(x)| \in O\big(m^{-(\eta - d)/2}\big).$$
Therefore,
$\mathbf P\big(B_j(F + s) \ne  B_j(F + w)\big) \in O(m^{-(\eta - d)/2}).$
Hence, noting that $\eta > d$ concludes the proof of the \eqref{eq:p3}.
\enp

%
%
\section{Proof of the FCLT, Theorem \ref{thm:fclt}}
\label{sec:prf-fclt}
In this section, we prove the functional CLT from Theorem \ref{thm:fclt}. After having established the fixed-level CLT  in Theorem \ref{thm:fix_clt}, we now need to prove tightness. 
We recall that from $ F\in \mathcal{C}^{l_0}$, we can conclude by Theorem  \ref{thm:gass}} that 
$\sup_{i \in \Z^d}\mathbf E\big(Y^\partial(B(i, 1) \ti \R)^{l_0}\big) < \ff$.

Henceforth, we write $Y = Y^{\partial }$   for the set critical and stratified critical points above level $ u$ of the field $F$, where the stratification is with respect to $W_n$, see Section  \ref{sec:morse-topo}.
For deriving the functional CLT, it will be essential to ensure that the decomposed Betti numbers $\beta_n^+$ and $\beta_n^-$ are both non-increasing in the level $u$. To ensure this, we recall from Proposition  \ref{prop:decompos-b-n} $ \beta _n = \beta _n^{ + }-\beta _n^{-}$ with
\begin{align*}\beta _n^{ \pm }(u;F) = \sum_{(x,v,C)}[\delta (x,v,C)]_{ \pm }
\end{align*}
Since the sum of tight processes is tight and symmetry considerations it suffices to prove the tightness statement when replacing $\beta_n$ by $\beta_n^{ + }$.

Set $\tbnp(I) := \beta^+_n(I) - \mathbf E(\beta^+_n(I))$. We prove tightness by verifying the Chentsov-condition from {\blue \cite[Theorem 13.5]{billingsley} (together with Markov's inequality)},
\begin{align}
        \label{chent_eq}
        \mathbf E\big(\tbnp(I)^4\big) \le c |W_n|^2|I|^{5/4},
\end{align}
for suitable $c > 0$, where for $I = [u_-, u_+]$, we set $\tbnp(I) := \tbnp(u_+) - \tbnp(u_-)$. The crucial tool is the cumulant expansion
\begin{align}
        \label{cex_eq}
        \mathbf E\big(\tbnp(I)^4\big) = 3\Var\big( \tbnp(I)\big)^2 + c_4\big( \tbnp(I)\big).
\end{align}
In Section \ref{sec:grid} below {\blue (more precisely, the second paragraph of that section)}, we rely on a trick from {\blue \cite[Theorem 2]{davydov}}. {\blue This trick allows to deduce tightness if the Kolmogorov-Chentsov condition \eqref{chent_eq} holds for \emph{$n$-big} intervals $I$, i.e., to intervals with $|I|^{} \ge |W_n|^{-2/3}$.} Hence, to establish condition \eqref{chent_eq}, we derive refined bounds on variances and cumulants.

%
%
\bepr
\label{vc_prop}
Let $l_0 > 16$ and $\g > 54d/(1 - 8/{l_0})$. Then,
 $$\sup_{n \ge 1}\sup_{I \text{ is $n$-big}}{|W_n|^{-1}|I|^{-5/8}\Var( \btt(I)\big) + |W_n|^{-7/6}c_4\big( \btt(I)\big)} < \ff.$$
\enpr

The variance and cumulant bounds in Proposition \ref{vc_prop} give the tightness.

\bep[Proof of \eqref{chent_eq}; $ n$-big intervals]
First, by combining Proposition \ref{vc_prop} and the identity
$$\mathbf E\big(\tbnp(I)^4\big) \le 3c|W_n|^2|I|^{5/4} +c |W_n|^{7/6}.$$
Since $|I|$ is $n$-big, we deduce that $|W_n|^{7/6} \le |W_n|^2|I|^{5/4}$, as asserted.  It remains to prove {\blue that Kolmogorov-Chentsov on $n$-big intervals implies tightness, see below (more precisely, the second paragraph of Section \ref{sec:grid})}.
\enp

After {\blue showing that Kolmogorov-Chentsov on $n$-big intervals implies tightness}, we prove Proposition \ref{vc_prop} in Section \ref{var_sec}. To achieve this goal, the key task is bound both the   moments for the critical points inside percolation sets, and also to to refine the moment condition so as to reflect the level. To allow for clear reference, we state the result as a separate lemma.

Let $\mc C_i = \mathcal{C}(\{F\ge u_c\};Q_i)$ denote the union of all connected components of $\{F \ge u_c\}$ intersecting the cube $Q_i$. {\blue Here, $u_c$ is the percolation level threshold from Section \ref{sec:gauss}.}
%
%

\bel[Moment bounds]
\label{lem:lev}
Let $l_0 \ge 2^{13}$.  Then,
\been
\im $\sup_{i \in \Z^d}\mathbf E\big(Y^\partial(\CC_i \ti \R)^{{l_0/2}} \big)< \ff$;
\im $\sup_{I \su \Ib}\sup_{i \in \Z^d}|I|^{-31/32}\mathbf E\big(Y^\partial(\CC_i \ti I)^{\sqrt{l_0/2}} \big) < \ff$ for every compact $\Ib \su \R$.
\enen
\enl
\bep
We prove the two parts separately.
\smallskip

\noindent {\bf Part (1).}
Let $K_i\ge 1$ be the smallest integer such that $\CC_i \su B(i, K_i)$. Then, by the Cauchy-Schwarz inequality,
\begin{align*}
\sup_{i\in \Z^d}\mathbf E\big(Y^\partial(\CC_i \ti \R)^{l_0/2} \big) &\le \sup_{i\in \Z^d}\sum_{k \ge 1}\mathbf E\big(Y^\partial(B(i, k) \ti \R)^{l_0/2} \one\{K_i = k\} \big)\\
	&\le \sup_{i \in \Z^d}\sum_{k \ge 1}\sqrt{\mathbf E\big(Y^\partial(B(i, k) \ti \R)^{l_0}\big)} \sqrt{\mathbf P(K_i = k)}\\
	&\le\sup_{i \in \Z^d} \sqrt{\mathbf E\big(Y^\partial(B(i, 1) \ti \R)^{l_0}\big)}\sum_{k \ge 1}(2k)^{1 + l_0/2} \sqrt{\mathbf P(K_i = k)}.
\end{align*}
Finally, the sum converges because of {\blue Assumption  \eqref{ass:perco}}.  Note that the first term is finite by applying Theorem \ref{thm:gass} to all faces of $B(i, 1)$.

\smallskip

\noindent{\bf Part (2).}
Set $M: = \sqrt{l_0/2}$.
We proceed similarly as in the proof of part (1). The difference is that instead of the Cauchy-Schwarz inequality, we use the H\"older inequality with $q = M$ and $p = M/(M - 1)$. Then,
\begin{align*}
\mathbf E\big(Y^\partial(\CC_i \ti I)^M \big) &\le \sum_{k \ge 1}\mathbf E\big(Y^\partial(B(i, k) \ti I)^M \one\{K_i = k\} \big)\\
	&\le  \sum_{k \ge 1}\big(\mathbf E\big(Y^\partial(B(i, k) \ti I)^{pM}\big)\big)^{1/p} {\mathbf P(K_i = k)}^{1/q}\\
	&\le  \big(\mathbf E\big(Y^\partial(B(i, 1) \ti I)^{pM}\big)\big)^{1/p}\sum_{k \ge 1}(2k)^{M+1} \mathbf P(K_i = k)^{1/q}.
\end{align*}
As in part (1), the convergence of the sum follows from the {\blue percolation Assumption \eqref{ass:perco}.}
To ease notation, write $Y^\partial$ short for $Y^\partial(B(i, 1) \ti I)$. Note that $(Y^\partial)^{pM}\le (Y^\partial)^{pM + p}$. Hence, another application of the H\"older inequality gives that

\begin{align*}
\mathbf E\big(Y^\partial(B(i, 1) \ti I)^{pM}\big)
&\le 	 \mathbf E(Y^\partial(B(i, 1)\ti \R)^{pM^2})^{1/M} \mathbf E(Y^\partial(B(i, 1)\ti I))^{1/p} \end{align*}
Since $ 2M^2 \le l_0$, we deduce that the first factor is finite, whereas the second is of order $O\big(|I|^{1/p}\big)$ by the Kac-Rice Lemma  \ref{lm:KR} applied to each facet of $B(i, 1)$. Since, $M \ge 64$, we conclude that $1/p^2 \ge 31/32$, thereby concluding the proof.
\enp

%
%
\subsection{Kolmogorov-Chentsov on $n$-big intervals implies tightness}
\label{sec:grid}

In this section, we show that {\blue Kolmogorov-Chentsov on $n$-big intervals implies tightness.} The key idea is to consider the approach from {\blue \cite[Theorem 2]{davydov}}. In order to be able to apply this result, we recall from Section 3 that $\btt(u)$ is decreasing in $u$.

\def\Y^\partials{Y^\partial^{strat}}

Now, {\blue \cite[Theorem 2]{davydov}} 
{\blue shows that Kolmogorov-Chentsov on $n$-big intervals implies tightness,} provided that $\mathbf E(\btt(I)) \in o\big(\sqrt{|W_n|}\big)$ holds for all $n$-small intervals $I\su \Ib$. {\blue Indeed, we can then apply \cite[Theorem 2]{davydov} with $a_n := |W_n|^{-2/3}$.}
Now, we show that $\mathbf E(\btt(I)) \in o\big(\sqrt{|W_n|}\big)$. We recall from Section 3 that $| \delta (x,v,C) | \le c \# Y^\partial\cap (C\ti [u,\ff ))$ for $v \ge u \ge u_c$. Therefore,    with $Y_j :=Y^\partial (Q_j \ti I)$,
$$\btt(I)\le c\sum_{j \in \bar W_n}Y^\partial\big(\mc C_j \ti [u_c,\ff)\big) \one\{Y_j \ne 0\}.$$
Hence, taking expectations and using Lemma \ref{lem:lev}, we arrive at
\begin{align}
	\label{eq:btty}
	\mathbf E\big(\btt(I)\big)\le c \sum_{j \in \bar W_n}\mathbf E\big(Y^\partial\big(\mc C_j \ti [u_c,\ff)\big)  \one\{Y_j \ne 0\}\big)\big)
\end{align}
Now, we bound the second summand in \eqref{eq:btty}. First, by the H\"older inequality,
	$$\mathbf E\big(Y^\partial\big(\mc C_j \ti [u_c,\ff)\big)  \one\{Y_j \ne 0\}\big)\big)\le\big(\mathbf E\big(Y^\partial\big(\mc C_j \ti [u_c,\ff)\big)^{32}\big)\big)^{\frac1{32}} \mathbf P\big(Y_j \ne0\big).^{\frac{31}{32}}$$
	The first factor is bounded by part (1) of Lemma \ref{lem:lev}. Finally, the second factor is of order $O(|I|^{31/32})$ by Lemma \ref{lm:KR}. Now, referring again to the smallness of $I$ concludes the proof.

%
%
\subsection{Proof of Proposition \ref{vc_prop}}
\label{var_sec}
To prove Proposition \ref{vc_prop}, we rely on the martingale technique that was already implemented in the setting of cylindrical networks \cite{cyl}. To make the presentation self-contained, we recollect here the basic set-up.

We let $\GG_j$ be the $\s$-algebra generated by the restriction of $\W$ to boxes of the form $Q_i$ with {\blue $i \preceq j$}. Then, setting $B_{+, j}^{(n)}(I) :=\btt(I; F) - \btt(I; F^{(Q_j)})$ as in \eqref{eq:bdej},  we decompose the centred increment $\tbnp(I)$ as
$$\btt(I) - \mathbf E(\btt(I)) =\sum_{j \in \Z^d} \mathbf E(B_{+, j}^{(n)}(I)\ba \GG_j ).$$
{\blue Note that for fixed $i_1$, the expression $\sum_{i_1 \preceq j \preceq i_2} \mathbf E(B_{+, j}^{(n)}(I)\ba \GG_j)$ is a forward martingale in $i_2$. In \eqref{eq:csho}, we will show that it is square-integrable so that it converges to $\btt(I) -    \mathsf E(\btt(I)\ba \GG_{i_1 - 1})$ by martingale convergence. Another application of martingale convergence then shows the convergence as $i_1 \to -\ff$. We refer to \cite[Section 3.1]{bel} for a detailed proof.
}

We now prove the variance bounds in Proposition \ref{vc_prop}.

%
%
\bep[Proof of Proposition \ref{vc_prop} -- variance]
The key observation from \cite{yukCLT} is that $\{B_{+, i}^{(n)}\}_{ i\in \Z^d}$ is a martingale-difference sequence because $\mathbf E(\btt(I; F^{(Q_j)}) \ba \GG_j) = \mathbf E(\btt(I; F) \ba \GG_{j - 1})$.    Hence, the Cauchy-Schwarz inequality implies that
$$	\Var( \btt(I)) = \sum_{j \in \Z^d}\Var\big(\mathbf E(B_{+, j}^{(n)}(I) \ba \GG_j )\big) \le  \sum_{j \in \Z^d}\mathbf E\big(B_{+, j}^{(n)}(I)^2\big).$$
We now let  $B_{+, i, j}$ denote the  { positive} contribution to $B_{+, j}$ coming from critical points in the cube $Q_i$.  {\blue More precisely, recall from Proposition \ref{prop:decompos-b-n} that     
\begin{align*}\beta^+ _{n}(I;F) = \sum_{(x,v)\in \Ys(W_n\times I)}\delta (x,v,C_x,F)_+
\end{align*}
Then, 
$B_{+, i, j} := \sum_{(x,v)\in \Ys(W_n\times I) \cap (Q_i \ti I)}\big(\delta (x,v,C_x,F)_+-\delta (x,v,C_x,F^{(Q_j)})_+ \big)$.
} Then, we have a decomposition
\begin{align}
	\label{eq:bdejp}
	B_{+, j}^{(n)}(I)  =  \sum_{i \in {\blue \bar W_n}}B_{+, i, j}(I).
\end{align}
Now,  let $M = 16$ and $p' =16/15$. Then, by Cauchy-Schwarz and H\"older,
\begin{align}
	\label{eq:csho}
	\mathbf E\big(B_{+, j}^{(n)}(I)^2\big)^{1/2}\hspace{-.2cm} &\le \hspace{-.2cm}\sum_{i \in {\blue \bar W_n}}\hspace{-.1cm}\hspace{-.1cm}\sqrt{\mathbf E\big(B_{+, i, j}(I)^2\big)} 
	\hspace{-.1cm}\le \hspace{-.2cm}\sum_{i \in {\blue \bar W_n}}\hspace{-.2cm} \mathbf E\big(B_{+, i, j}(I)^{2M}\big)^{(2M)^{-1}}\hspace{-.3cm}\mathbf P(B_{+, i, j}(I) \ne 0)^{(2p')^{-1}}
\end{align}
We claim that the first factor is of constant order. Indeed, we first note that $F$ and $F^{(Q_j)}$ have the same distribution. Moreover, by the weak Morse inequality, the Betti number of a component is bounded by the number of critical points. Finally, Lemma \ref{lem:mom1} implies that
$\mathbf E\big(Y^\partial(\CC_0 \ti \R)^{2M} \big) < \ff,$
thus showing the asserted finiteness of the first factor.

For the second factor, an application of \eqref{eq:bde} shows that
\begin{align*}
 \mathbf P\big(B_{+, i, j}(I) \ne 0\big)^{1/5}\mathbf P\big(B_{+, i, j}(I) \ne 0\big)^{4/5} \in O\big(|i - j|^{-\g d/5} |I|^{4/5}\big)
\end{align*}
Therefore,
$\mathbf E\big(B_{+,j}^{(n)}(I)^2\big)^{1/2} \le  O\big( \sum_{i \in {\blue \bar W_n}}|i - j|^{-\g d/8} |I|^{3/8}\big)$.
Hence, we conclude that if $d_{j, n}:=\dist(j, W_{2n}) \le n^{1/d}$, then $\mathbf E\big(B_{+, j}^{(n)}(I)^2\big) \in O(|I|^{3/4})$. Moreover, if $d_{j, n} \ge n^{1/d}$, then
$\mathbf E\big(B_{+, j}^{(n)}(I)^2\big) \in O(d_{j,n}^{2(1 - \g/8)d}|I|^{3/4}).$
Finally,
$\sum_{j \co d_{j, n} \ge n}d_{j,n}^{2(1 - \g/8)d} \in |W_n|^{3 - \g/4}.$
Thus, recalling the assumption $\g> 8$ concludes the proof.

\enp

Hence, it remains to establish the cumulant bounds in Proposition \ref{vc_prop}. To achieve this goal, we need to have a suitable control on the correlation structure. To make this precise, we let $\ddd(\bs z) := \max_{\{S, T\} \prec \{i, j, k, \ell\}} \ddd_{S, T}(\bs z)$, where the maximum is taken over all possible partitions of $\bs z = \{i, j, k, \ell\}\su\Z^d$ into two non-empty groups, and where $\ddd_{S, T}(\bs z) := \dist\big({S, T}\big)$ is the inter-partition distance, when partitioning $\bs z $ into two groups indexed by $S$ and $T$, respectively. We set $\bs \DD(I) := \big\{B_{+, i}^{(n)}(I), B_{+, j}^{(n)}(I), B_{+, k}^{(n)}(I), B_{+, \ell}^{(n)}(I)\big\}$. Again, we state the decorrelation property now, and postpone the proof.

%
%
\bel[Spatial decorrelation]
\label{dec_lem}
Let ${ l_0} > 2^{13}$ and $\g > 1$.
Let $\bs z \su \Z^d$ with $|\bs z| \le 4$. Then, for every $\e > 0$,
$$\big|c_4\big(B_{+, i}^{(n)}(I), B_{+, j}^{(n)}(I), B_{+, k}^{(n)}(I), B_{+, \ell}^{(n)}(I)\big)\big|\in O\big({|W_n|^4} \ddd(\bs z)^{-\g d(1 - 8/{l_0})}\big).$$
\enl
Relying on Lemma \ref{dec_lem}, we now complete the proof of the cumulant bound. To ease notation, we henceforth drop the dependence on $I$ in the quantity $B_{+, i}^{(n)}(I)$.

%
%
\bep[Proof of Proposition \ref{vc_prop} -- cumulants]
First, by multilinearity of cumulants,
\begin{align}
	\label{eq:ijkl}
	c_4\big(\tbnp(I)\big) = \sum_{i, j, k, \ell \in \Z^d}a_{i,j,k, \ell} c_4(B_{+, i}^{(n)}, B_{+, j}^{(n)}, B_{+, k}^{(n)}, B_{+, \ell}^{(n)}),
\end{align}
where the $a_{i, j, k, \ell} \ge 1$ are suitable combinatorial constants. They only depend on which indices $i, j, k, \ell$ are equal.

We now decompose \eqref{eq:ijkl} into contributions with indices in the set $\mc I := \{i \co d_{i, n} \le n^{1/d}\}$ and those  with some indices outside $\mc I$. We bound the contributions from $\mc I^c$ and $\mc I$ separately. In both cases, we use that each summand in \eqref{eq:ijkl} can be bounded as
\begin{align}
	\label{eq:ijkl2}
	|c_4(B_{+, i}^{(n)}, B_{+, j}^{(n)}, B_{+, k}^{(n)}, B_{+, \ell}^{(n)})| \in O\big(\prod_{m\in\{i, j, k, \ell\}}\mathbf E((B_{+, m}^{(n)})^4)^{1/4}\big).
\end{align}
Now, by Lemma \ref{lem:mom1}, we have
$\mathbf E(|B_{+, i}^{(n)}|^4)^{1/4} \in O\big(|W_n|^{5/4} d_{i,n}^{-\g d/5 }\big)$. In particular,
$$\sum_{i \co d_{i, n} > n^{1/d}}\mathbf E(|B_{+, i}^{(n)}|^4)^{1/4} \in O\big(|W_n|^{1 + 5/4  - \g/5}\big),$$
so that
\begin{align*}
	\sum_i\mathbf E(|B_{+, i}^{(n)}|^4)^{1/4}  &= \sum_{i : d_{i,n}\le n^{1/d}}\mathbf E(|B_{+, i}^{(n)}|^4)^{1/4}  + \sum_{i:d_{i,n}> n^{1/d}}\mathbf E(|B_{+, i}^{(n)}|^4)^{1/4} 
	\in O(  | W_n | ),
\end{align*}
and
\begin{align*}
	&\sum_{\mc I^c}|c_4(B_{+, i}^{(n)}, B_{+, j}^{(n)}, B_{+, k}^{(n)}, B_{+, \ell}^{(n)})|
	\le C\hspace{-.6cm}\sum_{i\co d_{i,n}> n^{1/d}}\hspace{-.4cm}\mathbf E(|B_{+, i}^{(n)}|^4)^{\frac14}\Big(\sum_{j \in\Z^d}\mathbf E(|B_{+, j}^{(n)}|^4)^{\frac14}\Big)^3
	\hspace{-.2cm}	\in O(|W_n|^{21/4  - \g/5}).
\end{align*}
 Hence, noting that $\g> 22$ shows that the last line is in $o(|W_n|)$.

Therefore, it remains to deal with indices in $\mc I$.
Here, we partition the sum over $\mc I$ again into two parts $\Sigma_1 + \Sigma_2$, where $\Sigma_1$ and $\Sigma_2$ are determined as follows. The part $\Sigma_1$ contains all summands corresponding to indices $(i, j, k, \ell) \in \mc I$ such that $\ddd(\{i, j, k, \ell\})\le |W_n|^{\e/d}$ with $\e= 1/18$, and $\Sigma_2$ contains the remaining summands.

The sum $\Sigma_1$ consists of $O(|W_n|^{1 + 3\e})$ summands each of which is of constant order, because of the  bound \eqref{eq:ijkl2}. Thus, $\Sigma_1 \in O\big(|W_n|^{7/6}\big)$. Second, we note that $\Sigma_2$ consists of $O(|W_n|^4)$ summands. Moreover, by Lemma \ref{dec_lem} each of them is in $O\big(|W_n|^{{4-\g\e(1 - 4/\sqrt{l_0})/(2d)}}\big)$. Thus, since {$\g\e(1 - 4/\sqrt{l_0}) \ge 14d$}, we conclude.
\enp

It remains to prove the spatial decorrelation asserted in Lemma \ref{dec_lem}.  To that end, we will proceed along the blueprint provided in \cite[Lemma 2]{cyl}.
\bep[Proof of Lemma \ref{dec_lem}]
{
We may assume that
$k_0:=\ddd(\bs z) := \dist\big(\{i, j\}, \{k, \ell\}\big)$. The other cases are similar but easier. To make the field-dependence clear, we write $B_{+, i}^{(n)}(\tilde F)$ when 
we compute the functional with respect to some perturbation $\tilde F$ of the original field $F$. 

Moreover, by the cluster-decomposition of the cumulant from \cite[Lemma 5.1]{barysh}, it suffices to show the claim when replacing the cumulant by 
$\Cov\big(B_{+, i}^{(n)}B_{+, j}^{(n)},B_{+, k}^{(n)}B_{+, \ell}^{(n)} \big).$
Letting $\bs z' =(i', j', k', \ell') \in \Z^{4d}$, this covariance decomposes as
$\sum_{\bs z'}\Cov\big(B_{+,i', i}B_{+,j', j},B_{+,k', k}B_{+, \ell', \ell} \big),$
 {where $B_{+,i',i}$ }denotes the contribution to $B_{+, i}^{(n)}$ coming from critical points in the cube $Q_{i'}$. Also note that the number of $i'$ such that $Q_{i'} \cap W_n \ne \es$ is of order $O(|W_n|)$.
 Recall that $H_{i,k}$ is the half-space of points closer to $k$ than $i$. Also recall the resample field $F^{(i,k)}$. Then, $F^{(i, k)}$ and $F^{(k, i)}$ are independent, which implies
\begin{align*}
	&\Cov\big(B_{+, i}^{(n)}B_{+, j}^{(n)},B_{+, k}^{(n)}B_{+, \ell}^{(n)} \big)\\
	&=\Cov\big(B_{+, i}^{(n)}B_{+, j}^{(n)}-B_{+, i}^{(n)}(F^{(i, k)})B_{+, j}^{(n)}(F^{(i, k)}),B_{+, k}^{(n)}B_{+, \ell}^{(n)} \big) \\
&\quad+ \Cov\big(B_{+, i}^{(n)}(F^{(i, k)})B_{+, j}^{(n)}(F^{(i, k)}),B_{+, k}^{(n)}B_{+, \ell}^{(n)} - B_{+, k}^{(n)}(F^{(k,i)})B_{+, \ell}^{(n)}(F^{(k,i)}) \big).
\end{align*}
We now explain how to bound the first of those covariances, noting that the arguments for the second are analogous. First,  by the Cauchy-Schwarz inequality,
\begin{align*}
&\Big|\Cov\big(B_{+,i', i}B_{+,j', j} - B_{+,i', i}(F^{(i,k)})B_{+, j', j}(F^{(i,k)}),B_{+,k', k}B_{+, \ell', \ell}  \big)\Big| \\
	&\le\sup_{m', m \in \Z^d}\sqrt{\mathbf E\big((B_{+,m, m'})^4\big)} 
	\sqrt{\mathbf E\big(\big(B_{+,i', i}B_{+,j', j} - B_{+,i', i}(F^{(i,k)})B_{+,j', j}(F^{(i,k)})\big)^2 \big)}.
\end{align*}By the moment bound in Lemma \ref{lem:lev}, the second factor remains bounded.
We write $J_2$ for the second factor. 
Then, by the H\"older inequality with $q' = M > 1$ and $p' = M/(M - 1)$, 
\begin{align*}
	J_2&\le2 \mathbf P\big(\big(B_{+,i', i}B_{+,j', j}\ne
		  B_{+,i', i}(F^{(k,i)})B_{+,j', j}(F^{(k,i)}) \big)^{1/(2p')} 
	 \hspace{-.4cm}\sup_{m, m' \in \Z^d}\mathbf E\big(\big|B_{+,m, m'}\big|^{4M}\big)^{1/(2M)}.
\end{align*}
Setting $M:={l_0}/8$, we obtain $\mathbf P\big(B_{+,i', i} \ne B_{+, i', i}(F^{(i,k)})\big) \in O(\ddd(z)^{-\g d + \e})$ by  Lemma  \ref{lm:proba-topo},   and the arguments for $j$ are analogous.
 }
\enp

 \section*{Acknowledgements}
	{ We are thankful to the two anonymous referees for their very careful reading of the manuscript and many useful suggestions to improve the quality of the manuscript.} We also thank  F.~Severo and I.~Wigman for very helpful suggestions and references to literature, as well as  Stephen Muirhead, who also helped detect problems in the first version of the article.
This research was conducted while the second author was in Paris, supported by the {\it P\lowercase{ROGRAMME d'iNVITATIONS INTERNATIONALES SCIENTIFIQUES,
CAMPAGNE 2025}} from Université Paris Cité.
CH was supported by a research grant (VIL69126) from VILLUM FONDEN
\bibliographystyle{abbrv} 
\bibliography{../../lit}

\end{document}